\documentclass{amsart}
\usepackage[utf8]{inputenc}
\usepackage{amssymb,amsmath,amsthm}
\usepackage{amsfonts}
\usepackage{amsmath, amssymb}
\usepackage{graphicx} 
\usepackage[font=small,labelfont=bf]{caption}
\usepackage{epstopdf}
\usepackage[colorlinks, linkcolor=black, citecolor=blue, urlcolor=blue]{hyperref}
\usepackage{xcolor}
\usepackage{amsthm}
\usepackage{pgfplots}
\usepackage{listings}
\usepackage{longtable}
\usepackage{mathrsfs}
\usepackage[utf8]{inputenc}  
\usepackage[T1]{fontenc}
\usepackage{indentfirst} 
\usepackage{enumerate} 
\usepackage{fancyhdr} 
\usepackage{pdfpages} 
\usepackage{verbatim} 
\usepackage{float} 
\usepackage[numbers]{natbib}

\newtheorem{teo}{Theorem}
\newtheorem{prop}[teo]{Proposition}
\newtheorem{lema}[teo]{Lemma}
\newtheorem{cor}[teo]{Corollary}
\theoremstyle{definition} 

\newtheorem{obs}[teo]{Remark}

\DeclareMathOperator{\V}{SL}

\newcommand{\seq}{\xrightarrow{\tau}}
\newcommand{\N}{_n(\mathbb{R})}
\newcommand{\dif}{\mathop{}\!\mathrm{d}}
\newcommand{\NN}{(\mathbb{R})}

\newcommand{\K}{\mathcal{K}^n}

\newcommand{\RR}{(\mathbb{R}^n)}
\DeclareMathOperator{\conv}{conv}
\DeclareMathOperator{\Conv}{Conv}

\DeclareMathOperator{\I}{I}
\DeclareMathOperator{\CD}{Conv_{cd}}

\DeclareMathOperator{\PA}{Conv_{p.a}}
\DeclareMathOperator{\LQ}{_{l.q}}

\DeclareMathOperator{\dom}{dom}
\DeclareMathOperator{\interior}{int}

\DeclareMathOperator{\D}{D^2}

\DeclareMathOperator{\epi}{epi}
\DeclareMathOperator{\LC}{Conv_{ld}}

\DeclareMathOperator{\conc}{Conc([0,\infty))}
\DeclareMathOperator{\proj}{proj}

\author[Fernanda M. Baêta]{Fernanda Moreira Baêta}
\address{\tiny{INSTITUT FÜR DISKRETE MATHEMATIK UND GEOMETRIE, TECHNISCHE UNIVERSITÄT WIEN, WIEDNER HAUPTSTRASSE 8-10/1046, 1040 WIEN, AUSTRIA}}
\email{fernanda.baeta@tuwien.ac.at}
\subjclass[2020]{52B45, 52A41, 26A51, 26A16}
\title[Upper semicontinuous valuations on convex functions]{Upper semicontinuous valuations on convex functions of one variable}
\pgfplotsset{compat=1.18} 
\begin{document}
\pretolerance10000
\begin{abstract}
A classification of upper semicontinuous, translation  and dually epi-translation invariant valuations   is established on the space of convex Lipschitz function on $\mathbb{R}$ with compact domain.
\end{abstract}
\maketitle

\section{Introduction  and statement of the main result}
Let $\mathcal{K}^n$ be the space of  \textit{convex bodies}, that is,  non-empty compact convex sets in  $n$-dimensional Euclidean space $\mathbb{R}^n$.  A functional $Z:\mathcal{K}^n\rightarrow\mathbb{R}$ is called a \textit{valuation} if
$$Z(K\cup L)+Z(K\cap L) = Z(K)+Z(L)$$
whenever $K,L,K\cup L\in\mathcal{K}^n$. The most familiar examples of valuations are the intrinsic volumes  $V_0, \dots,V_{n}$, which include the $n$-dimensional volume, $V_n$, the Euler characteristic, $V_0$, and the surface area, $2V_{n-1}$. They  have useful properties such as: continuity with respect to the Hausdorff metric, invariance under rigid motions, monotonicity and homogeneity (see, for example, \cite{Vitor}). Valuations have many applications in integral geometry (see \cite{klain}) and  play an important role in convex geometry (see \cite{mullen, mullen2}). 

Let $\V\N$ denote the special linear group of $\mathbb{R}^n$.  We say that a valuation $Z$ is invariant under volume preserving maps if $Z(\phi K +x)= Z(K)$ for every vector $x\in\mathbb{R}^n$, every $ \phi\in \V\N$, and every convex body $K\in\K$. Specifically, when this equality holds only for rotations  $\phi$, we say that $Z$ is a rigid motion invariant valuation. We consider continuous and upper semicontinuous valuations, where $\mathcal{K}^n$ and its subspaces are equipped with the topology induced by the Hausdorff metric.

The interest in classifying valuations on $\mathcal{K}^{n}$ originates with Blaschke \cite{Blaschke1937} and was later developed by Hadwiger. Probably the most famous result in this area  is the Hadwiger characterization theorem (1957), which classifies all continuous and rigid motion invariant valuations on the space $\mathcal{K}^n$ (see \cite{81}). Specifically, it states that any such valuation can be written as a linear combination of the intrinsic volumes $V_0, V_1, \dots, V_n$. In 1999, Ludwig and Reitzner established  an affine version of Hadwiger’s theorem,  providing a classification of upper semicontinuous valuations that are  invariant under volume preserving maps (see \cite{71', monika}).  Further characterizations of valuations with interesting invariance properties were obtained, for example, in \cite{alesker,haberl, haberl2, 81,klain,L, mullen,mullen2}.

The following result,  due to    Ludwig \cite{71},  classifies  rigid motion invariant and upper semicontinuous valuations defined on $\mathcal{K}^2$. Consider the set
\begin{multline*}
\conc=\\
\left\{\zeta: [0,\infty)\rightarrow [0,\infty)\mid \ \zeta \ \mbox{is concave}, \ \lim_{t\rightarrow 0}\zeta(t)=0,  \mbox{ and } \lim_{t\rightarrow +\infty}\zeta(t)/t =0\right\}.    
\end{multline*}

\begin{teo}[\cite{71}]\label{generalize}
A functional  $Z:\mathcal{K}^2\rightarrow \mathbb{R}$ is an upper semicontinuous and rigid motion invariant valuation if and only if there exist constants $c_0,c_1,c_2\in\mathbb{R}$ and a function $\zeta\in \conc$ such that 
\begin{align*}
 Z(K) = c_0V_0(K) + c_1V_1(K)+c_2 V_2(K) + \int_{S^1} \zeta(\rho(K,u))\dif\mathcal{H}^1(u)  
\end{align*}
for every $K\in\mathcal{K}^2$.
\end{teo}
\noindent Here,  $S^1$ is the unit sphere in $\mathbb{R}^2$, while $\rho(K,u)$ is the curvature radius of the boundary of $K$ at the point with normal $u\in S^1$, and $\mathcal{H}^1$ is the 1-dimensional Hausdorff measure.

Recently, the notion of valuations has been extended to families of functions due to their close relationship with valuations on convex bodies. If $X$ is a space of extended real-valued functions,  a function $Z: X\to \mathbb{R}$ is a valuation if 
$$Z(u\wedge v)+ Z(u\vee v)= Z(u) + Z(v)$$
for every $u,v\in X$ such that also $u\wedge v, u\vee v\in X$, where $u\wedge v$ and $u\vee v$ denote the
pointwise  minimum and maximum of $u$ and $v$, respectively. Characterization results for valuations on various function spaces have been provided, including spaces of convex functions \cite{ irmn, H1, M1}, Lebesgue $L_p$ spaces \cite{TS}, and Sobolev spaces \cite{ludwig2}. In this paper, we focus on real-valued valuations defined on  spaces of convex functions.

We say that a function $u:\mathbb{R}^n\rightarrow(-\infty,+\infty]$ is \textit{lower semicontinuous} if its epigraph is a closed subset of $\mathbb{R}^{n+1}$, and it is a \textit{proper} function if the (effective) \textit{domain} of $u$,
$$\dom u =\{x\in\mathbb{R}^n\mid \ u(x)<+\infty\},$$
is non-empty. We denote by 
$$\Conv(\mathbb{R}^n)=\{u: \mathbb{R}^n\rightarrow (-\infty, +\infty]\mid \ u \  \mbox{is l.s.c and convex}, \ u\not\equiv +\infty\} $$
the space of  proper, lower semicontinuous, convex functions defined on $\mathbb{R}^n$, and  by 
\begin{align*}
\Conv(\mathbb{R}^n;\mathbb{R})= \{v:\mathbb{R}^n\to \mathbb{R}\mid \ v \text{ is convex}\}    
\end{align*}
the space of finite-valued, convex functions on $\mathbb{R}^n$.

Throughout the text, we will say that $u\in \Conv(\mathbb{R}^n)$ is Lipschitz on $\interior(\dom u)$ and, for short, we say just Lipschitz, if there exists a constant $L_u>0$ such that 
\begin{align}\label{3..}
 |u(x)-u(y)|\leq L_u\|x-y\|   
\end{align}
for every $x,y\in \interior (\dom u)$, where    $\interior (A) $  denotes the interior a set $A\subset\mathbb{R}^n$, and   $\|y\|$ denotes the Euclidean norm of $y$. We consider the subset of functions
\medskip
\begin{align*}
 \LC\RR=
\{u\in \Conv(\mathbb{R}^n)\mid \ \dom u \mbox{ is compact},\, u \mbox{ is Lipschitz on }\interior (\dom u)\}.
\end{align*}

\noindent Note that if $u \in \Conv(\mathbb{R}^n)$ and $\dom u$ is compact but not full-dimensional, then $u \in \LC\RR$ without necessarily being Lipschitz.

We equip $\LC\RR$ with the following notion of convergence (cf.~\cite{integral}):  We say that $u_k\seq u$ (or that $u_k$ is $\tau$-convergent to $u$) if  $u_k,u\in \LC\RR$ and
\begin{enumerate}
\item[(i)] for every sequence $x_k$ that converges to $x$, we have $$u(x)\leq \liminf_{k\rightarrow +\infty} u_k(x_k);$$
\item[(ii)] there exists a sequence $x_k$ that converges to $x$ such that 
$$u(x)= \lim_{k\rightarrow +\infty}u_k(x_k);$$
\item[(iii)] letting $L_{u_k}$ denote the Lipschitz constant of $u_k$  on the interior of its domain, the sequence $\{L_{u_k}\}$ is uniformly bounded by some constant $M>0$, independent of $k$.
\end{enumerate}
\noindent  The convergence (i) and (ii) is called \textit{epi-convergence}. Moreover, if $u\in \Conv(\mathbb{R}^n)$ and  $\dom u$  has non-empty interior, 
then $u_k\in \Conv(\mathbb{R}^n)$ epi-converges to $u$ if and only if $u_k$ converges uniformly to $u$ on every every compact set that does not contain  a boundary point of $\dom u$ (see [\citealp{76}, Theorem 7.17]). Condition (iii) for the space of real-valued Lipschitz continuous maps defined on the unit sphere $S^{n-1}$
has been considered  in  \cite{topology, topology2,topology3},  with the additional requirement of uniform convergence on the unit sphere.

A valuation $Z:\LC\RR\rightarrow \mathbb{R}$ is called \textit{translation invariant}  if 
$$Z(u\circ \tau)= Z(u)$$
for every $u\in \LC\RR$ and translation $\tau$ in $\mathbb{R}^n$, and  $\V\N$ \textit{invariant} if
$$Z(u\circ \phi)=Z(u)$$
for every $u\in \LC\RR$ and $ \phi\in \V\N$. We define $Z: \LC\RR\rightarrow \mathbb{R}$ to be $\tau$-\textit{upper semicontinuous} if, for every sequence $u_k$ that is  $\tau$-convergent to $u$,
$$Z(u)\geq \limsup_{k\rightarrow +\infty} Z(u_k).$$
We say that $Z$ is \textit{dually translation invariant} if  $Z(u+l)= Z(u)$ for every linear function $l$ and for every $u\in \LC\RR$, and  \textit{vertically translation} invariant if $Z(u+c)=Z(u)$ for every $u\in \LC\RR$ and every constant $c$. When $Z$ is both dually translation invariant and vertically translation invariant, we say that $Z$ is \textit{dually epi-translation} invariant.

Let $D\subset \mathbb{R}^n$ be a convex set. For a convex function $u:D\rightarrow (-\infty, +\infty]$, let  $ \D u(x)$ denote the Hessian matrix of $u$ at $x$ if $u$ is twice differentiable at $x$ and $\D u(x)=0$ otherwise.  By  Aleksandrov \cite{alek}, a convex function $u$ is twice differentiable almost everywhere.  The functional 
\begin{align}\label{func}
    Z(u)=\int_{D}(\det\D u(x))^{\frac{1}{n+2}}\dif x,
\end{align}
is known as the \emph{affine area functional} (see \cite{TW2,TW6,TW8}), 
because for bounded $D$ it equals the affine surface area of the 
graph of $u$.  Note that  $\zeta(t)=t^{1/(n+2)}$ belongs to $\conc$, and the functional $Z$  defined in \eqref{func} is  an $\V\N$, translation and dually epi-translation invariant valuation, and by \cite[Theorem 1]{integral}, it is $\tau$-upper semicontinuous.  For affine surface area of convex bodies, see  \cite{monika, lutwak, floating}; for a recent survey, see \cite{E}. By [\citealp{TW8}, Lemma 2.3], the functional
\begin{align*}
    u\mapsto \int_{D}\det\D u(x)\dif x,
\end{align*}
is the absolutely continuous part of the well-known  \textit{Monge--Ampère measure} on $D$, with respect to the Lebesgue measure (see, for example,  \cite{99}).

Our goal is to establish a functional version of Theorem~\ref{generalize}. 
While Theorem~\ref{generalize} focuses on valuations on convex bodies in the two-dimensional setting, 
we now consider convex functions on $\mathbb{R}$, that is, the case $n=1$. 
Note that although the variable is one-dimensional, the epigraph of a convex function in $\LC\NN$ is a convex subset of $\mathbb{R}^2$, so from a geometric perspective we are still dealing with a two-dimensional problem.
However, the functional framework does not allow us to exploit additional structures such as invariance under $\mathrm{SL}_2(\mathbb{R})$. 
This makes classification in the one-dimensional functional setting more delicate; for example, previous results (see, e.g., \cite{irmn,M1}) require $n\ge 2$ precisely for these reasons.  The following theorem provides a classification result in this setting.

\begin{teo}\label{maintheorem}
A functional  $Z:\LC\NN\rightarrow\mathbb{R}$ is  a $\tau$-upper semicontinuous, translation and dually epi-translation invariant valuation  if and only if there are constants $c_0,c_1\in\mathbb{R}$ and a  function   $\zeta\in\conc$ such that
\begin{align}\label{A}
Z(u)= c_0+c_1V_1(\dom u)+\int_{\dom u}\zeta(u''(x))\dif x
\end{align}
for every $u\in \LC\NN$.
\end{teo}

In particular, Theorem~\ref{maintheorem} also provides a characterization of the more general functional
\begin{align}\label{integral_m}
Z(u) = \int_{\dom u} \zeta(u''(x))\dif x,
\end{align}
for $\zeta \in \conc$ and $u \in \LC\NN$, where we set $u''(x)=0$ whenever $u$ is not twice differentiable at $x$. In Remark \ref{obs}, we explain why the restriction to sequences with uniformly bounded Lipschitz constants is necessary.

Since the integral in \eqref{integral_m} vanishes for piecewise affine functions, we obtain the following consequence.
\begin{cor}
A functional  $Z:\LC\NN\rightarrow\mathbb{R}$ is  a $\tau$-continuous, translation and dually epi-translation invariant valuation  if and only if there are constants $c_0,c_1\in\mathbb{R}$ such that
\begin{align*}
Z(u)= c_0+c_1V_1(\dom u)
\end{align*}
for every $u\in \LC\NN$.
\end{cor}
Note that  $u\mapsto V_1(\dom u)$ is not continuous with respect to epi-convergence (see Remark \ref{obs}).

\section{Tools}\label{preliminaries}

In this section we present some basic results and establish some notation that will be used throughout the paper. We also show that each term that appears in  \eqref{A} satisfies the conditions of being a $\tau$-upper semicontinuous, translation and  dually epi-translation invariant valuation.  In particular, if $u_k$ is $\tau$-convergent to $u$, then the volume of the domain depends continuously on $u$. In this section, we present the results in $\mathbb{R}^n$ to provide a framework that can be used in future work.

For a function $u \in \Conv(\mathbb{R}^n)$, we denote by $\mbox{epi}(u)$ its epigraph, 
$$\mbox{epi}(u) = \{(x,t)\in \mathbb{R}^{n+1} \mid \ u(x)\leq t\},$$
and by $\conv(A)$ the convex hull of a set $A \subset \mathbb{R}^n$, 
which is  the smallest  convex set containing $A$.

Consider the indicator functions of convex bodies $K \in \mathcal{K}^n$, defined by
\begin{align*}
\I_K(x)=\begin{cases}
0,&  \ \mbox{if} \ \ x\in K\\
+\infty,&  \ \mbox{if} \ \ x\not\in K.
\end{cases}
\end{align*}
Note that these indicator functions belong to $\LC\RR$.

The following lemma is a simple result that allows us  to obtain new valuations from a given one.
\begin{lema}\label{lemmasum}
Let $Z:\LC\RR\rightarrow \mathbb{R}$ be a valuation. If $w\in \LC\RR$ and
\begin{align*}
Z_{w}(u):=Z(u+w),
\end{align*}
for $u\in \LC\RR$, then $Z_w$ is a valuation on $\LC\RR$.
\end{lema}

\begin{proof}
Let $w,u,v$  be convex functions in $\LC\RR$ such that $u\wedge v\in \LC\RR$ as well. It is easy to see that
\begin{align*}
u\wedge v +w &=  (u+w)\wedge (v +w)\\
u\vee v +w &=  (u+w)\vee (v +w).    
\end{align*}
Applying $Z$ to the equations above and using that $Z$ is a valuation, we obtain
\begin{align*}
Z_w( u\wedge v) + Z_w( u\vee v) &=  Z( u\wedge v +w)+ Z( u\vee v +w )\\
&= Z( (u+w)\wedge (v +w))+ Z( (u+w)\vee (v +w) )\\
&= Z( u+w)+ Z( v+w )\\
&=Z_w( u)+ Z_w( v ).
\end{align*}
Therefore, $Z_w$ is a valuation.
\end{proof}

Before stating the next lemma, we recall the definition of Minkowski addition for sets in $\mathbb{R}^n$: for $A, B \subset \mathbb{R}^n$, we define
$$A + B := \{ a + b \mid a \in A, \, b \in B\}.$$

\begin{lema}\label{domain}
Let $u, v \in \LC\RR$, and assume that $u \wedge v \in \LC\RR$. Then the following equalities hold:
\begin{align*}
\dom(u \vee v) &= \dom u \cup \dom v,\\
\dom(u \wedge v) &= \dom u \cap \dom v.
\end{align*}
In particular,  the Minkowski addition satisfies
$$\dom(u\vee v) + \dom(u\wedge v) = (\dom u \cup \dom v) + (\dom u \cap \dom v).$$
\end{lema}

\begin{proof}
Let $u,v\in \LC\RR$ be such that $ u\wedge v\in \LC\RR$. Since
\begin{align*}
\dom{u\vee v}&=\{x\in\mathbb{R}^n\mid \ u(x)<+\infty \ \mbox{and} \ v(x)<+\infty\} \\
&= \{x\in\mathbb{R}^n: u(x)<+\infty\} \cap \{x\in\mathbb{R}^n\mid \ v(x)<+\infty\}\\
&= \dom u\cap \dom v
\end{align*}
and
\begin{align*}
\dom{u\wedge v}&=\{x\in\mathbb{R}^n\mid \ u(x)<+\infty \ \mbox{or} \ v(x)<+\infty\} \\
&= \{x\in\mathbb{R}^n\mid \ u(x)<+\infty\} \cup \{x\in\mathbb{R}^n: v(x)<+\infty\}\\
&= \dom u\cup \dom v,
\end{align*}
this yields the result.
\end{proof}

\begin{lema}\label{volume2}
The functional $Z:\LC\RR\rightarrow \mathbb{R}$ defined by
\begin{align*}
Z(u)= V_n(\dom u)
\end{align*}
is an $\V\N$, translation and dually epi-translation invariant valuation.
\end{lema}

\begin{proof}
Let $u,v\in\LC\RR$ be such that $ u\wedge v\in\LC\RR$ as well. Using that  $V_n$ is a valuation and  Lemma \ref{domain}, we obtain
\begin{align*}
Z(u\vee v) & = V_n(\dom u\vee v)\\
& = V_n(\dom u\cap \dom v)\\
& = V_n(\dom u) + V_n(\dom v)-V_n(\dom u\cup \dom v)\\
& = V_n(\dom u) + V_n(\dom v)-V_n(\dom u\wedge v)\\
&= Z(u)+ Z( v) - Z(u\wedge v),
\end{align*}
i.e., $Z$ is a valuation. Now let $c \in \mathbb{R}$, let $l:\mathbb{R}^n \to \mathbb{R}$ be a linear function and  $\tau_y(x)=x+y$ for $y\in\mathbb{R}^n$. Since $\dom(u+l+c)=\dom u$, $\dom(u\circ \tau_y^{-1})= \dom u+y$ and $V_n$ is a translation invariant valuation, it follows that
$$Z(u+l+c)= V_n(\dom(u+l+c))= V_n(\dom u)= Z(u)$$
and
$$Z(u\circ \tau_y^{-1})= V_n(\dom(u\circ \tau_y^{-1}))= V_n(\dom u+y)= V_n(\dom u)= Z(u).$$
Thus, $Z$ is a translation and dually epi-translation invariant valuation. To show that $Z$ is $\V\N$ invariant, we  only need to observe that for $\varphi\in \V\N$, it holds that $\dom u\circ\varphi=\varphi^{-1}(\dom u)$, and that $V_n$ is $\V\N$ invariant.
\end{proof}

To say that $u_k\in \Conv(\mathbb{R})$ epi-converges to $u\in \Conv(\mathbb{R})$ is equivalent to saying that $\epi(u_k)$ converges to $\epi(u)$  in the Painlevé-Kuratowski sense as $k\rightarrow +\infty$  (see [\citealp{76},  Definition 7.1, Proposition 7.2]). By \cite[Theorem 4.26]{76}, if $ F:\mathbb{R}^{n+1}\to \mathbb{R}^n$ is a continuous mapping and  $u_k$ epi-converges to $u$, i.e.,  
$$ \limsup_{k\rightarrow +\infty}\epi(u_k)= \liminf_{k\rightarrow +\infty} \epi(u_k)=\epi(u),$$ 
then we obtain the inclusion 
\begin{align*}
    F(\liminf_{k\to +\infty} \epi(u_k))\subseteq \liminf_{k\to +\infty} F(\epi(u_k)).
\end{align*}
Taking $F(x, t)= x$, which maps epigraphs to  domains, we obtain
\begin{align*}
\dom u= F(\epi (u)) =F(\liminf_{k\to +\infty} \epi(u_k))\subseteq \liminf_{k\to +\infty} F(\epi(u_k))= \liminf_{k\to +\infty} \dom u_k.     
\end{align*}
By definition, 
\begin{align*}
 \liminf_{k\to +\infty}   \dom u_k \subseteq  \limsup_{k\to +\infty}   \dom u_k.
\end{align*}
Hence,
\begin{align}\label{inc}
\dom u \subseteq \liminf_{k\rightarrow +\infty}\dom u_k \subseteq\limsup_{k\rightarrow +\infty}\dom u_k,    
\end{align} 
(see also \cite[Propostion 7.4]{76}). The definitions of $\liminf_{k\rightarrow +\infty}C_k$ and $\limsup_{k\rightarrow +\infty}C_k$ for a sequence of sets $C_k\subset\mathbb{R}^n$ can be found in [\citealp{76}, Chapter 4]. Furthermore, by [\citealp{76}, page 117], if $C_k$ converges to $C$ in the Painlevé-Kuratowski sense  and there exists a bounded set $B\subset\mathbb{R}^n$ such that $C_k,C\subset B$, then this convergence is equivalent to convergence in the Hausdorff metric.  In order to prove Lemma \ref{volume1}, we will need the following results from \cite{76}.

\begin{teo}[\cite{76}, Theorem 4.5 (b)]\label{sup}
For sets $C_k,C\subset \mathbb{R}^n$ with $C$ closed, we have 
\begin{align*}
    \limsup_{k\to \infty} C_k\subset C
\end{align*}
if and only if for every compact set $B\subset \mathbb{R}^n$ with $C\cap B=\emptyset$, there exists a set  $N\subset \mathbb{N}$ such that $\mathbb{N}\setminus N$ is finite  and $C_k\cap B=\emptyset$ for all $k\in N$. 
\end{teo}
\medskip

\begin{teo}[\cite{76}, Corollary 4.12]\label{cor}
Let $C_k\subset \mathbb{R}^n$ be compact and connected sets,  with   $\limsup_{k\to \infty} C_k$ bounded. Then there exists a bounded set $B\subset \mathbb{R}^n$ such that $C_k\subset B$ for all $k$.
\end{teo}
\medskip

\begin{lema}\label{volume1}
The functional $Z:\LC\RR\rightarrow \mathbb{R}$ defined by
\begin{align*}
Z(u)= V_n(\dom u)
\end{align*}
is $\tau$-continuous.
\end{lema}

\begin{proof}
To prove that $Z$ is a $\tau$-continuous valuation for $n\geq1$, let $u_k\in \LC\RR$  be such that  $u_k\seq u$. First, consider the case
$V_n(\limsup_{k\rightarrow +\infty} \dom u_k) = 0$. 
In this case, $V_n(\dom u_k) \to 0$ as $k \to +\infty$, and by \eqref{inc} we have
$$V_n(\dom u) \le V_n\Big(\limsup_{k\to +\infty} \dom u_k\Big) = 0 = \lim_{k\to +\infty} V_n(\dom u_k),$$
so the continuity of $Z$ follows. Otherwise, assume that $\limsup_{k\rightarrow  +\infty}\dom u_k$ has non-empty interior and that the Lipschitz constants of $u_k$ are uniformly bounded by some constant $L>0$. We first analyze the convergence of the domains in the Painlevé sense.  By \eqref{inc},  it suffices to show that
\begin{align*}
 \limsup_{k\rightarrow +\infty}\interior(\dom u_k)= \limsup_{k\rightarrow +\infty}\dom u_k\subseteq \dom u.   
\end{align*}
The first equality is established in \cite[Proposition 4.4]{76}. We proceed by contradiction and suppose that
\begin{align}\label{not}
\limsup_{k\rightarrow +\infty}\interior(\dom u_k)\not\subset \dom u.    
\end{align} 
Then, by Theorem \ref{sup}, there exists a compact set $B\subset\mathbb{R}^n$ such that $\dom u\cap B=\emptyset$, and for every set of indices $N\subseteq\mathbb{N}$, where $\mathbb{N}\setminus N$ is finite,  there exists some $k\in N$ such that $\interior(\dom u_k)\cap B\not=\emptyset$. Now, define the following sets of indices
\begin{align*}
    N_0&= \{k\geq 1\},
\end{align*}
and for $j\geq1$,
\begin{align*}
N_j&=\{k> m_j\}, \quad \mbox{where} \qquad m_j=\min\{k\in N_{j-1}:\interior(\dom u_k)\cap B\not=\emptyset\}.
\end{align*}
We must have $m_j \neq \emptyset$ for every $j\in\mathbb{N}$, since $\mathbb{N}\setminus N_{j-1}$ is finite, \eqref{not} holds, and Theorem \ref{sup} applies. Moreover,  $m_{j+1}>m_j$ for every $j\in\mathbb{N}$.
\medskip

Now, consider the subsequence  $\{u_{m_j}\}_j\subseteq \{u_k\}_k$. Since $u_k$ epi-converges to $u$, we also have that $u_{m_j}$ epi-converges to $u$ as $j\rightarrow +\infty$. Let $x_{m_j}\in\interior(\dom u_{m_j})\cap B$ for every $j\in\mathbb{N}$. Since $\{x_{m_j}\}_j\subset B$ and  $B$ is compact, the sequence $\{x_{m_j}\}_j$ admits a convergent subsequence, say $\{x_{m_{j_k}}\}_k$, with $x_{m_{j_k}}\rightarrow x$ as $k\rightarrow +\infty$, for some $x\in\mathbb{R}^n$. Since $B$ is closed, we also have $x\in B$. Now, by the definition of epi-convergence and using that $\dom u\cap B=\emptyset$, we get
\begin{align}\label{cd}
   + \infty= u(x)\leq \liminf_{k\rightarrow +\infty} u_{m_{j_k}}(x_{m_{j_k}}).
\end{align}

Next, let $y_0\in \dom u$. By the definition of epi-convergence, there exists a sequence $y_{m_{j_k}}$ converging to $y_0$ such that
$$\lim_{k\rightarrow +\infty}u_{m_{j_k}}(y_{m_{j_k}})=u(y_0)<+\infty.$$
Then, for $\varepsilon=1$, there exists $N>0$ such that for $k>N$,
\begin{align}\label{3.}
    |u_{m_{j_k}}(y_{m_{j_k}})|\leq |u(y_0)|+1.
\end{align}
Let $C$ be a compact set containing the sequences $\{x_{m_{j_k}}\}_k\subset B$ and $\{y_{m_{j_k}}\}_k\subset \dom u$. Such a set exists because both $B$ and $\dom u$ are compact. Since each $u_{m_{j_k}}$ is a Lipschitz function with Lipschitz constant bounded above by $L$, and using \eqref{3..} and \eqref{3.}, we  obtain
\begin{align*}
    |u_{m_{j_k}}(x_{m_{j_k}})|& \leq L \, \|x_{m_{j_k}}-y_{m_{j_k}}\|+|u_{m_{j_k}}(y_{m_{j_k}})|\\
    & \leq L\,\mbox{diam}(C) + |u(y_0)|+1\\
    &< M,
\end{align*}
for each $k>N$, where diam$(C)$ is the diameter of $C$ and $M>0$ is a constant that does not depend on $k$. However, this contradicts \eqref{cd}. Therefore, no such compact set $B$ can exist, and we conclude that
\begin{align*}
\limsup_{k\rightarrow +\infty}\dom u_k\subset \dom u\subseteq\liminf_{k\rightarrow +\infty}\dom u_k    \subseteq\limsup_{k\rightarrow +\infty}\dom u_k.
\end{align*}
This implies that $\dom u_k$ converges to $\dom u$ in the Painlevé-Kuratowski sense.

To conclude, by Theorem \ref{cor},  there exists a bounded set containing both $\dom u$  and $\dom u_k$ for all $k\in\mathbb{N}$.
Since the domains of $u_k$ converge to $\dom u$ in the Painlevé sense and all sets are contained in the same bounded set, it follows that $\dom u_k$ converges to $\dom u$ in the Hausdorff metric (see \cite[page 117]{76}).  
Moreover, the $n$-dimensional volume $V_n$ is continuous with respect to the Hausdorff metric, so we have
$$V_n(\dom u_k) \to V_n(\dom u) \quad \text{as } k \to +\infty.$$
Therefore, the desired continuity of $Z$ is established, and the proof is complete.
\end{proof}

We now extend the functional \eqref{integral_m} to the general $n$-dimensional setting
\[
Z(u) = \int_{\dom u} \zeta(\det \D u(x))  \dif x, 
\]
for $\zeta \in \conc$ and $u\in \LC(\mathbb{R}^n)$, which coincides with \eqref{integral_m} when $n=1$. 
\medskip

The following result can be found in [\citealp{integral}, Theorem 1].

\begin{teo}\label{mainteo}
Let $\zeta\in\conc$. Then
\begin{align*}
    Z(u)=\int_{\dom u}\zeta ( \det\D u(x))\dif x
\end{align*}
is finite for every $u\in \LC\RR$ and depends $\tau$-upper semicontinuously on $u$. 
\end{teo}

\begin{obs}\label{obs}
In Theorem \ref {maintheorem} we consider Lipschitz functions because the functional 
$$\int_{\dom u}\zeta(u''(x))\dif x$$ 
is not finite for every $u\in\LC\NN$ and every $\zeta\in\conc$. An example is  $u(x)=-\sqrt{x}+\I_{[0,1]}(x)$ for $x\in\mathbb{R}$ and $\zeta(t)=t^{2/3}$ for $t\geq 0$. 

Furthermore, with respect to epi-convergence, consider the sequence
$$u_k(x) = kx^2 + \I_{[0,1]}(x), \quad k \in \mathbb{N}.$$ 
We have that $u_k \in \CD\NN$ epi-converges to $u = \I_{\{0\}}$ as $k \to +\infty$, but 
$$\int_{\dom u_k} \zeta(u_k''(x))  \dif x$$
does not converge, while the corresponding integral for the limit function vanishes. 
Note that the Lipschitz constants of the sequence $\{u_k\}_k$ are not uniformly bounded. Moreover,  $V_1(\dom u)$ is not continuous with respect to epi-convergence.  For instance, the sequence $u_k$ above epi-converges to $\I_{\{0\}}$, 
but $V_1(\dom u_k)$ does not converge to $V_1(\dom u)$ (cf.\ Example 1 in \cite{integral}).
\end{obs}

The following result establishes key properties of $Z$ in this more general framework.

\begin{lema}\label{concave} 
For $\zeta\in\conc$, the function $Z:\LC\RR\rightarrow [0,\infty)$ defined by $$Z(u)=\int_{\dom u}\zeta(\det\D u(x))\dif x,$$ is an $\V\N$, translation and dually epi-translation invariant valuation. \end{lema}

\begin{proof}
Using  Lemma \ref{domain}, we get  
\begin{align*}
&Z(u)+Z(v)  = \int_{\dom u}\zeta(\det\D u(x))\dif x + \int_{\dom v}\zeta( \det\D v(x))\dif x\\
& =\int_{\dom u\setminus \dom v}\zeta(\det\D u(x))\dif x+\int_{\dom u\cap \dom v}\zeta(\det\D u(x))\dif x   \\
& \quad + \int_{\dom v\setminus \dom u}\zeta(\det\D v(x))\dif x+\int_{\dom u\cap \dom v}\zeta(\det\D v(x))\dif x\\
& = \left(\int_{\dom u\setminus \dom v}\zeta(\det\D (u\wedge v)(x))\dif x+ \int_{\dom v\setminus \dom u}\zeta(\det\D(u\wedge v)(x))\dif x\right. \\
& \quad  \left.+ \int_{\dom u\cap \dom v}\zeta(\det\D (u\wedge v)(x))\dif x \right)+ \int_{\dom u\cap \dom v}\zeta(\det\D(u\vee v)(x))\dif x\\
& = \int_{\dom u\wedge v}\zeta(\det\D(u\wedge v)(x))\dif x+\int_{\dom u\vee v}\zeta(\det\D (u\vee v)(x))\dif x\\ \medskip
& = Z(u\wedge v) +Z(u\vee v)
\end{align*}
whenever $u,v,u\wedge v\in \LC\RR$. Thus,  $Z$ is a valuation.
\medskip

Now let   $\tau_y(x)=x+y$, where $y\in\mathbb{R}^n$. Then
\begin{align*}
\int_{\dom(u\circ \tau_y^{-1})}\zeta(\det\D(u\circ \tau_y)(x))\dif x &= \int_{\dom u-y}\zeta(\det\D u(x+y))\dif x\\
&= \int_{\dom u}\zeta(\det\D u(x))\dif x,    
\end{align*}
and for  an affine  function $w$, we have
$$\int_{\dom u}\zeta(\det\D(u+w)(x))\dif x = \int_{\dom u}\zeta(\det\D u(x))\dif x.$$
Furthermore, for $\varphi\in\V\N$ we obtain
\begin{align*}
Z(u\circ\varphi)=  \int_{\varphi^{-1}(\dom u)}\zeta(\det\D u\circ\varphi(x))\dif x&= \int_{\varphi^{-1}(\dom u)}\zeta(\det\D u(\varphi(x)))\dif x\\
&= \int_{\dom u}\zeta(\det\D u(x))\dif x\\
&= Z(u).    
\end{align*}
This concludes the proof.
\end{proof}

We say that a valuation $Z: \LC\RR\rightarrow \mathbb R$ is \textit{simple} if $Z(u)=0$ for every function $u\in \LC\RR$ such that $\dom u$ is not full dimensional. The following lemma will be used to prove Theorem \ref{maintheorem}.
\begin{lema}\label{non-positive}
Let $Z: \LC\RR\rightarrow \mathbb R$ be  a simple, $\tau$-upper semicontinuous, $\V\N$, translation and dually epi-translation invariant valuation. Then there exists a constant $c_1\in\mathbb{R}$ such that for every polytope $P\subset \mathbb{R}^n$ it holds
\begin{align*}
   Z(\I_P) = c_1 V_n(P).
\end{align*}
\end{lema}
\begin{proof}
Since $Z$ is a simple valuation, it suffices to prove the result for $n$-dimensional polytopes and, since every polytope $P$ can be dissected into simplices, i.e., the convex hull of $n+1$ points in $\mathbb{R}^n$,
it is enough  to prove the result for simplices. Let  $S$ be a simplex. Since $Z$ is simple, $\V\N$ and translation invariant valuation, the value $Z(\I_S)$ depends only on the volume of $S$. Thus, we can define the function $g:[0,+\infty)\rightarrow \mathbb{R}$  by
$$g(x)= Z(\I_S),$$
where $x=V_n(S)$ is the volume of the simplex $S$. Now, let $x_1,x_2\geq 0$ be the volumes of two simplices  $S_1$ and $S_2$, respectively. Assume that $S_1\cup S_2$ is  also a simplex with volume $x_1+x_2$.
Since $Z$ is a simple valuation, we have the additivity property 
$$Z(\I_{S_1\cup S_2})= Z(\I_{S_1})+Z(\I_{S_2}),$$
which implies that
$g(x_1+x_2)=g(x_1)+g(x_2)$. Thus, the function $g$ satisfies Cauchy's functional equation. Since $Z$ is $\tau$-upper semicontinuous, it follows that $g$ is upper semicontinuous, i.e,
$$\limsup_{k\to \infty} g(x_k)\le g(x)$$
whenever $x_k\to x$.  By a standard result (see, for example, \cite{aczel}),  there exists a constant $c_1$ such that $g(x)=c_1x$, and consequently
$$Z(\I_S)=c_1V_n(S)$$
for every simplex $S$. 
\end{proof}

Let $P, P_1,\dots, P_m\subset\mathbb{R}^n$ be polytopes  such that  $P=\bigcup_{i=1}^m P_i$ is also a polytope, and the interiors of the $P_i's$ are pairwise disjoint. A function $u\in \LC\RR$ belongs to  $\PA(\mathbb{R}^n)$  if there exist polytopes $P_1,\dots, P_m$ as before such that 
\begin{align*}
u=\bigwedge_{i=1}^m(w_i+\I_{P_i}),    
\end{align*}
where $w_i:\mathbb{R}^n\rightarrow \mathbb{R}$ are affine functions for each $i=1,\dots, m$.  By [\citealp{1}, Corollary 12],  there exists a sequence $u_k\in\PA\RR$ that epi-converges to $u\in \LC\RR$. Moreover, we can choose this sequence $u_k$  such that the Lipschitz constant of the $u_k's$
are bounded by the Lipschitz constant of $u$. Therefore, we conclude that $u_k\in\PA\RR$ and   $u_k\seq u$, and we obtain the following result.
\begin{lema}\label{piecewise}
For every $u\in \LC\RR$, there exists a sequence $u_k\in \PA(\mathbb{R}^n)$ that is $\tau$-convergent to $u$.
\end{lema}


To prove Proposition \ref{prop1} in Section \ref{sec3}, we will use the following version of the Vitali covering theorem.   A collection of sets $\mathcal{C}$ is called a \emph{Vitali class} for $S \subset \mathbb{R}^n$ if, for each $x \in S$ and each $\delta > 0$, there exists $U \in \mathcal{C}$ such that $x \in U$ and $0 < V_n(U) \le \delta$.

\begin{teo}[\cite{2}, Theorem 1.10]\label{VCT}
Let $S \subset \mathbb{R}^n$ be a bounded set, and let $\mathcal{C}$ be a Vitali class for $S$. Then for any $\varepsilon > 0$, there exists a finite collection of disjoint sets $U_1, \dots, U_m \in \mathcal{C}$ such that
\[
V_n\Big(S \setminus \bigcup_{i=1}^m U_i \Big) < \varepsilon.
\]
\end{teo}

\section{Proof of  Theorem \ref{maintheorem}}\label{sec3}
We now focus on the one-dimensional case ($n=1$) to prove Theorem \ref{maintheorem}. 
Specifically, we aim to show that if $Z:\LC\NN\rightarrow\mathbb{R}$ is a $\tau$-upper semicontinuous, translation   and dually epi-translation invariant valuation, then there exist  constants $c_0,c_1\in\mathbb{R}$ and a function  $\zeta\in\conc$ such that
\begin{align*}
Z(u)= c_0+c_1V_1(\dom u)+\int_{\mathbb{R}}\zeta(u''(x))\dif x
\end{align*}
for every $u\in \LC\NN$. We will adapt elements from the proof of the main theorem in  \cite{71}.  As a first step, we focus on the case where $Z$ is a simple, $\tau$-upper semicontinuous, translation and dually epi-translation invariant valuation that vanishes on indicator functions of closed intervals.

By Lemma \ref{piecewise}, every function in $\LC\NN$ can be approximated by  functions  $u_k\in \PA(\mathbb{R}^n)$. Since  $Z$ is dually epi-translation invariant, vanishes on indicator functions of closed intervals, and is  a $\tau$-upper semicontinuous valuation, we obtain
\begin{align}\label{znon_negative}
    Z(u)\geq 0
\end{align}
for every $u\in \LC\NN$.

Given $m>0$,    define the function $\zeta:[0,\infty)\rightarrow \mathbb{R}$ by
\begin{align}\label{eq4}
\zeta(a)= \frac{1}{2m}Z(aq +\I_{[-m,m]}),
\end{align}
where  
\begin{align}\label{q}
q(x)= \frac{x^2}{2}    
\end{align}
for $x\in\mathbb{R}$ and $a>0$. Throughout the text, $q$ will denote the quadratic function defined in \eqref{q}. Moreover, for a  closed interval $J \subset \mathbb{R}$, define the function $g_a : [0,\infty) \rightarrow \mathbb{R}$ by
\begin{align*}
g_a(x) = Z(a q + \I_J), \quad x = V_1(J),
\end{align*}
where $V_1(J)$ denotes the length of $J$. Note that in this case, the function $g_a$ depends on $a$, and   it is well-defined since $Z$ is simple, translation and  dually epi-translation invariant. 

By Lemma \ref{lemmasum} and the fact that $Z$ is simple, we have that
\begin{align*}
    g_a(x_1+x_2)= g_a(x_1)+g_a(x_2),
\end{align*}
whenever  $x_1,x_2\geq 0$. Since $Z$ is $\tau$-upper semicontinuous, it follows that $g_a$ is upper semicontinuous. Therefore, $g_a$ is a solution to Cauchy's functional equation, and   there exists a constant $c=c(a)$ such that
$$g_a(x)=cx.$$
Thus, we obtain
$$Z(aq+ \I_J)= cV_1(J)$$
for every closed interval $J\subset \mathbb{R}$. Combining this with \eqref{eq4}, we deduce
\begin{align}
\label{eq9a} Z(aq+ \I_{J})&= \frac{Z(aq+ \I_{[-m,m]})}{2m}V_1(J)\\
& =\zeta(a)V_1(J).\label{eq9}
\end{align}

The following  lemmas establish key properties of the function $\zeta$ defined in \eqref{eq4}, which will be crucial in the proof of Proposition~\ref{prop}.   The following result is a consequence of \eqref{znon_negative}.
\begin{lema}
Let $\zeta$ be defined by \eqref{eq4}. Then $\zeta$ is non-negative.\label{zetapositive}
\end{lema}

\begin{lema}\label{lemma1a}
Let $\zeta$ be defined by \eqref{eq4}. Then $\zeta$ is concave.
\end{lema}

\begin{proof}
To prove that $\zeta$ is concave, we use a geometric construction.  Let $$0\leq r<a<s,$$ 
and define  
$$q^a(x)=ax^2.$$  
Considering the points
$$p_i= -m+\left(\frac{2m}{n}\right)i,$$
where $p_0=-m,p_n=m$, and for $i=1,\dots, n-1$, we have $p_i\in (-m,m)$.
\medskip

Next, we define the family of functions
\begin{align*}
q_i^r(x)&= rx^2+(2a-2r)p_ix-(a-r)p_i^2,
\end{align*}
for $i=0,\dots, n-1$. We claim the following properties
$$q_i^r(p_i)=q^a(p_i), \quad (q_i^r)'(p_i)= (q^a)'(p_i) \quad \mbox{and} \quad (q_i^r+\I_{[-m,m]})(x)\leq (q^a+\I_{[-m,m]})(x),$$
for every $i=0,\dots, n-1$ and $x\in\mathbb{R}$. 
The first two properties ensure that the quadratic function $q_i^r$ matches the function $q^a$ at $x=p_i$ both in value and in derivative, while the third property ensures that $q_i^r$ lies below $q^a$ on the entire real line, but constrained by the interval $[-m,m]$. 
\medskip

Now, for each $i=1,\dots, n$, consider the function 
$$q_{i}^s(x)= sx^2+\beta_{i}x+\gamma_{i},$$
where $\beta_i$ and $\gamma_i$ are constants that are chosen such that the following conditions hold
\begin{align}\label{eq34}
\begin{cases}q_{i}^s(x_i) &= q_{i-1}^r(x_i)  \\
 (q_{i}^s)'(x_i) &= (q_{i-1}^r)'(x_i)  
 \end{cases}
\end{align}
for  some $x_i \in (p_{i-1},p_i)$, and  we also require that 
$$(q_{i-1}^r+\I_{[-m,m]})(x)\leq (q_{i}^s+ \I_{[-m,m]})(x)$$ 
for every $x\in\mathbb{R}$, which means that the function $q_{i}^s$ lies above $q_{i-1}^r$.

Similarly, we require the following for each $i=1,\dots, n$
\begin{align}\label{eq35}
   \begin{cases} q_{i}^s(y_i) &= q_{i}^r(y_i) \\
  (q_{i}^s)'(y_i) &= (q_{i}^r)'(y_i)
  \end{cases}
\end{align}
for  some $y_i \in (x_i,p_i)$, and  again we need 
$$(q_{i}^r+\I_{[-m,m]})(x)\leq (q_{i}^s+\I_{[-m,m]})(x)$$ 
for every $x\in\mathbb{R}$.

By  the second equations in \eqref{eq34} and  \eqref{eq35}, we have the following expressions for $x_i$ and $y_i$, respectively,
\begin{align}
\label{eq***}  x_{i}&= \left(\dfrac{a-r}{s-r}\right)p_{i-1}-\dfrac{\beta_i}{2s-2r} \\
 \label{eq****} y_{i}&= \left(\dfrac{a-r}{s-r}\right)p_i-\dfrac{\beta_i}{2s-2r}.
\end{align}
Thus
\begin{align}\label{yx}
y_{i}-x_{i}= \left(\frac{a-r}{s-r}\right)\frac{2m}{n},
\end{align}
which shows that  $y_i-x_i$ is a constant  that does not depend  on $i$. Moreover
\begin{align*}
x_{i+1}-y_{i}= \dfrac{\beta_{i}}{2s-2r} -\dfrac{\beta_{i+1}}{2s-2r}.
\end{align*}

From the first equation in \eqref{eq34} and  \eqref{eq***}, we obtain the condition
\begin{align}\label{eqt}
\left(\beta_i-(2a-2r)p_{i-1}\right)^2-4(s-r)\left(\gamma_i+(a-r)p_{i-1}^2\right)=0
\end{align}
and from the first equation in \eqref{eq35} and \eqref{eq****}, we get
\begin{align}\label{eqtt}
\left(\beta_i-(2a-2r)p_i\right)^2-4(s-r)\left(\gamma_i+(a-r)p_i^2\right)=0.
\end{align}
Using \eqref{eqt} and \eqref{eqtt}, we deduce that
\begin{align*}
\beta_i-\beta_{i+1}=\frac{4m}{n}((s-r)-(a-r)),
\end{align*}
which leads to the following expression for the difference between $x_{i+1}$ and $y_i$
\begin{align}\label{xiy}
\nonumber x_{i+1}-y_i&= \dfrac{1}{2s-2r}\left(\frac{4m}{n}((s-r)-(a-r))\right)\\
&= \dfrac{2m}{n}\left(1-\left(\frac{a-r}{s-r}\right)\right),
\end{align}
for every $i=1,\dots, n-1$. This means  that $x_{i+1}-y_i$  does not depend on $i$.
\medskip

We approximate the  function $q^a+\I_{[-m,m]}$ by a convex   function $v_n\in \LC\NN$, constructed as follows
\begin{multline*}
    v_n= (q_0^r+\I_{[-m,x_1]})\wedge (q_1^s+ \I_{[x_1,y_1]})\wedge(q_1^r+\I_{[y_1,x_2]})\wedge (q_2^s+\I_{[x_2,y_2]})\\
    \wedge \cdots \wedge (q_n^r+\I_{[y_n,m]}).
\end{multline*}
\begin{figure}[H]
\centering
\includegraphics[scale=0.12]{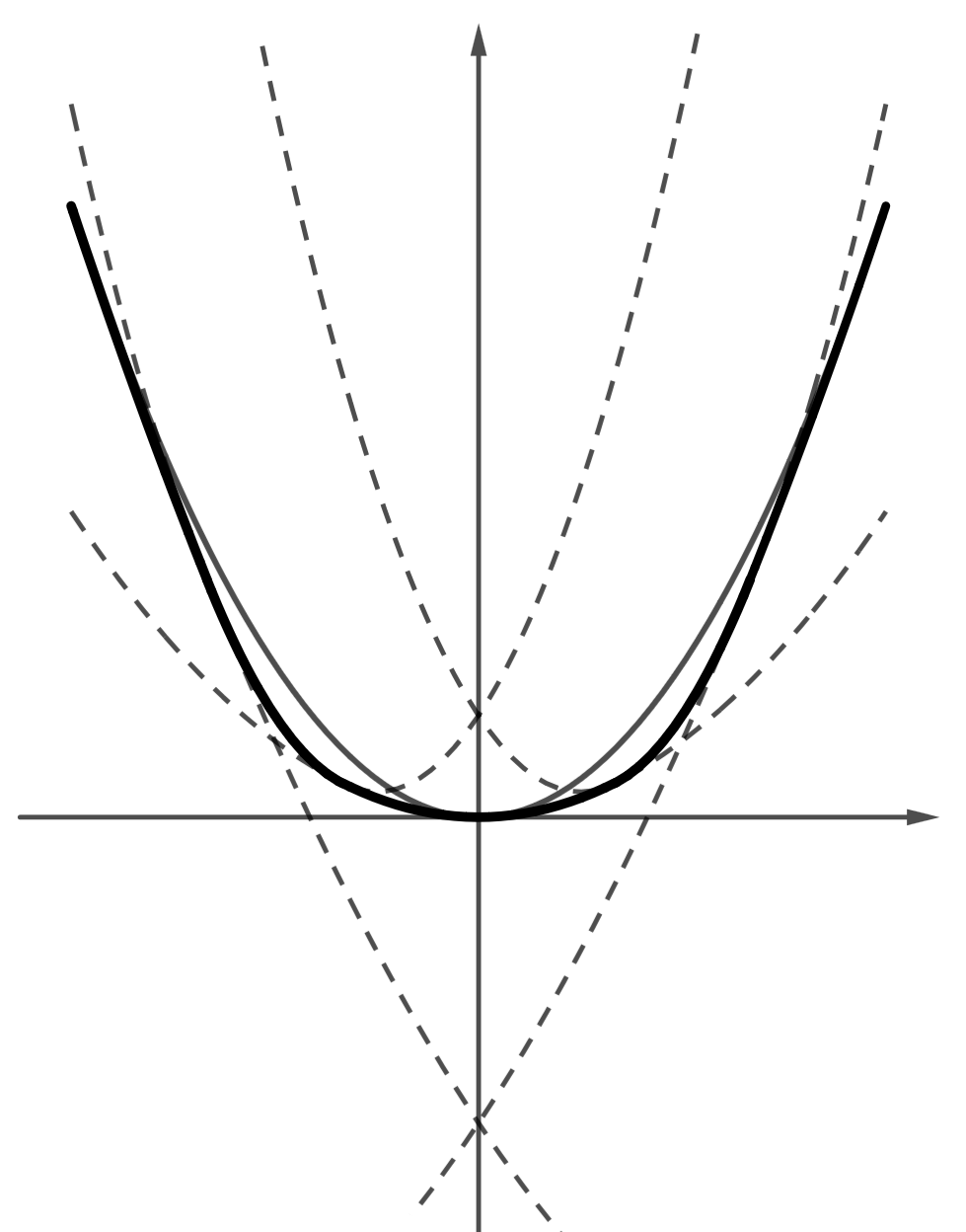}
\caption{Quadratic function $q^a$ (gray) and the corresponding approximation $v_n$ (black).
}
\end{figure}

Since $y_i-x_i$ and $x_{i+1}-y_i$ do not depend on $i$, and since $Z$ is a simple and dually epi-translation invariant valuation, we have
\begin{align}
\nonumber Z(v_n) &= n Z(q_1^s+\I_{[x_1,y_1]})+ nZ(q_1^r+\I_{[y_1,x_2]}),
\end{align}
and by \eqref{eq9a}, \eqref{yx}, and \eqref{xiy}, we get
\begin{align}
\nonumber Z(v_n)&= n(y_1-x_1)\dfrac{Z(q_1^s+\I_{[-m,m]})}{2m}+n(x_2-y_1)\dfrac{Z(g_1^r+\I_{[-m,m]})}{2m}\\
\label{eq5.10} &= \left(\frac{a-r}{s-r}\right)Z(q_1^s+\I_{[-m,m]})+ \left(1-\left(\frac{a-r}{s-r}\right)\right)Z(q_1^r+\I_{[-m,m]}).
\end{align}

Note that $Z(v_n)$ does not depend on $n$. To finish the proof  that $\zeta$ is concave, we  use  the fact that $v_n$ is  $\tau$-convergent to $q^a+\I_{[-m,m]}$ as $n\rightarrow +\infty$,  that $Z$ is $\tau$-upper semicontinuous and use  equation \eqref{eq5.10}, such that
\begin{align*}
2m \zeta(2a)&= Z(q^a+\I_{[-m,m]})\\
&\geq \limsup_{n\rightarrow+\infty}Z(v_n)\\
&= \limsup_{n\rightarrow+\infty} \left(\left(\frac{a-r}{s-r}\right)Z(q_1^s+\I_{[-m,m]})+ \left(1-\left(\frac{a-r}{s-r}\right)\right)Z(q_1^r+\I_{[-m,m]})\right)\\
&= 2m\left(\left(\frac{a-r}{s-r}\right)\zeta(2s)+ \left(1-\left(\frac{a-r}{s-r}\right)\right)\zeta(2r) \right).
\end{align*}
Setting $\lambda= \frac{a-r}{s-r}$, since $0\leq r<a<s$, we have that $0< \lambda<1$, and
$$\zeta(\lambda 2s + (1-\lambda)2r)\geq \lambda \zeta(2s) + (1-\lambda)\zeta(2r).$$
By the arbitrariness of $r$ and $s$, we conclude that
$$\zeta(\lambda s + (1-\lambda)r)\geq \lambda \zeta(s) + (1-\lambda)\zeta(r)$$
for every $\lambda\in (0,1)$, which shows that $\zeta$ is concave.    
\end{proof}
\medskip

\begin{lema}\label{lemma1}
Let $\zeta$ be defined by \eqref{eq4}. Then $\zeta \in \conc$.
\end{lema}

\begin{proof}
By Lemma \ref{lemma1a}, we already know that $\zeta$ is concave.  Consider the sequence of functions
$$u_k(x)=\dfrac{1}{k}x^2, \quad k\in \mathbb{N}.$$
Note that $u_k+\I_{[-m,m]}\in  \LC\NN$ and $u_k+\I_{[-m,m]}\seq \I_{[-m,m]}$ as $k\rightarrow +\infty$. 

By Lemma \ref{zetapositive}, $\zeta$ is non-negative, hence
$$\limsup_{a\rightarrow 0^+}\zeta(a)=\limsup_{k\rightarrow +\infty}\dfrac{Z(u_k+\I_{[-m,m]})}{2m}\leq Z(\I_{[-m,m]})=0.$$

To establish the final property needed for $\zeta \in \conc$, consider
$$\tilde{u}_k(x)=\frac{k}{2}x^2$$
for $k\in\mathbb{N}$. Note that $\tilde{u}_k+\I_{\left[0, \frac{1}{k}\right]}$ epi-converges to $\I_{\{0\}}$ as $k\rightarrow +\infty$, and for every $k\in\mathbb{N}$, the Lipschitz constant of $\tilde{u}_k+\I_{\left[0, \frac{1}{k}\right]}$ is 1. Since $Z$ is a $\tau$-upper semicontinuous and simple valuation, we get

\begin{align}\label{l1}
\limsup_{k\rightarrow +\infty} Z(\tilde{u}_k+\I_{\left[0, \frac{1}{k}\right]})\leq Z(\I_{\{0\}})= 0.
\end{align}

By Lemma \ref{zetapositive},  \eqref{eq9}, and since $Z$ is dually epi-translation invariant,  we obtain
\begin{align}\label{l2}
Z(\tilde{u}_k+\I_{\left[0, \frac{1}{k}\right]})= \dfrac{Z(kq + \I_{[-m,m]})}{2m}\frac{1}{k}= \dfrac{\zeta(2k)}{k}.
\end{align}
Thus, by \eqref{l1} and \eqref{l2}, we conclude that 
\[
\lim_{a\to+\infty} \frac{\zeta(a)}{a} = 0,
\]
which completes the proof that $\zeta \in \conc$.
\end{proof}

In the next step, we are interested in showing that $\zeta$ uniquely determines a valuation $Z$ that is simple, $\tau$-upper semicontinuous, translation and dually epi-translation invariant, and vanishes on indicator functions of closed intervals.



A function $u\in \LC\NN$ is called \textit{piecewise linear-quadratic} if $\dom u$ can be expressed as the union of finitely many intervals $J_i, i=1,\dots, l$, such that the restriction of $u$ to $J_i$ is either a quadratic or an affine function. The set of piecewise linear-quadratic  functions will be denoted by $P\LQ(\mathbb{R})$. Note that piecewise affine functions  belong to $\PA(\mathbb{R})$, and since  $\PA(\mathbb{R})$ is dense in $\LC\NN$ (see Lemma \ref{piecewise}),  every $u\in \LC\NN$ can be approximated by elements of $P\LQ(\mathbb{R})$.  The $\tau$-upper semicontinuity of $Z$ implies that for every sequence $u_k\in P\LQ(\mathbb{R})$ such that $u_k\seq u$,
\begin{align}\label{eq12}
Z(u)\geq \limsup_{k\rightarrow +\infty} Z(u_k).
\end{align}
We will prove that for every $u\in \LC\NN$, there exists a sequence $u_k\in P\LQ(\mathbb{R})$ such that  equality holds in \eqref{eq12}, i.e.,
\begin{align}\label{newe}
Z(u)=\sup\left\{\limsup_{k\rightarrow +\infty }Z(u_k)\mid \ u_k\in P\LQ(\mathbb{R}), u_k\seq u\right\}.
\end{align}
By proving this result, we will show that $Z$ is uniquely determined by $\zeta$. 
\medskip


We call a closed triangle $T=T(x,y)$ a \textit{support triangle} of a convex function $u$ with  endpoints $(x,u(x))$ and $(y,u(y))$, where  $x,y\in\mathbb{R} $, if  $T$ is bounded by support lines (i.e., 1-dimensional support hyperplanes) to $u$ at $x$ and $y$ and the chord connecting $(x,u(x))$ and $(y,u(y))$.
Using suitable support triangles of $u\in \LC\NN$, we will construct a function $v\in P\LQ(\mathbb{R})$ such that
\begin{align*}
    Z(u)\leq Z(v)+\rho V_1(\dom u),
\end{align*}
where $\rho>0$ is a given constant. Specifically, we will show that for the set $N$ of normal points, that is, 
$$N=\{x\in \dom u\mid \ \ u \mbox{ is twice differentiable at }x\},$$ 
there exists  a suitable Vitali class defined using the support triangles of $u$.  

Since Z is a
translation invariant valuation, we may assume without loss of generality that 
$$\dom u=[-m,m].$$ 
Denote by  $\proj_{e_1} C$  the projection of the set $C\subset \mathbb{R}^2$ onto the subspace  spanned by the canonical vector $e_1\in\mathbb{R}^2$, and let  $|\cdot|$ denote the absolute value of a scalar. We  use the notation $[A,B]$ for the closed line segment with endpoints $A$ and $B$, and $l_{[A,B]}$ for the linear function whose graph contains the segment $[A,B]$.
\begin{lema}\label{lemma2}
Let  $x_0\in N$ be such that $u''(x_0)>0$. For every $\rho, \delta>0$,  there exist a support triangle $T=T(x,y)$ of $u$, a convex body $S_T\subset \mathbb{R}^2$, and a function  $v_T\in P\LQ(\mathbb{R})$ such that:
\begin{enumerate}
\item [(i)]$(x_0,u(x_0))\in T$ and $0< y-x <\delta$;
\item[(ii)]$S_T\subset T$ and $S_T$ is a support triangle of $v_T$;
\item [(iii)]$Z(u+\I_{[x,y]})\leq Z(v_T)+\frac{\rho}{2} (y-x)$.
\end{enumerate}
\end{lema}
\begin{proof}
First, since $u$ is a  convex function,  it is,  in particular,  twice  differentiable  almost everywhere. By Taylor expansion, $u$   can be  locally represented around $x_0$ as 
\begin{align}\label{eq19}
u(x)=u(x_0)+u'(x_0)(x-x_0)+\frac{1}{2}u''(x_0)(x-x_0)^2+o((x-x_0)^2).
\end{align}

For $\varepsilon>0$, let  $A_\varepsilon=(x_0-\varepsilon,u(x_0-\varepsilon))$ and  $B_\varepsilon=(x_0+\varepsilon,u(x_0+\varepsilon))$  be points on the graph of $u$ and let 
$$T_\varepsilon=T_\varepsilon(x_0-\varepsilon,x_0+\varepsilon)$$ 
denote the support triangle of $u$ with endpoints $A_\varepsilon$ and $B_\varepsilon$. By \eqref{eq19}, we have
$$B_\varepsilon-A_\varepsilon= \left(2\varepsilon, 2\varepsilon u'(x_0)+o(\varepsilon^2)\right).$$
Thus, for $\varepsilon>0$ sufficiently small, we have  $0<2\varepsilon<\delta$, and condition $(i)$ holds. To prove condition $(ii)$, we now consider the support lines of $u$ at $x_0-\varepsilon$ and $x_0+\varepsilon$, denoted by $H(A_\varepsilon)$ and $H(B_\varepsilon)$,  respectively, and let $C_\varepsilon=(c_\varepsilon^1,c_\varepsilon^2)$ be the point where $H(A_\varepsilon)$ and $H(B_\varepsilon)$ intersect. Without loss of generality, assume that
$$ |\proj_{e_1} (C_\varepsilon - B_\varepsilon)| = (x_0+\varepsilon)-c_\varepsilon^1\leq c_\varepsilon^1-(x_0-\varepsilon)=|\proj_{e_1} (C_\varepsilon - A_\varepsilon)|.$$  
Now, define $B^1_\varepsilon=(b^1_\varepsilon, \tilde{b}^1_\varepsilon)$ as the point on $H(B_\varepsilon)$ such that 
$$|\proj_{e_1} (C_\varepsilon - B^1_\varepsilon)|= b^1_\varepsilon-c_\varepsilon^1=c_\varepsilon^1-(x_0-\varepsilon)= |\proj_{e_1} (C_\varepsilon - A_\varepsilon)|.$$ 
Since $B^1_\varepsilon\in H(B_\varepsilon)$, the point $C_\varepsilon$ also belongs to $H(B_\varepsilon)$, and 
$$ |\proj_{e_1} (C_\varepsilon - B_\varepsilon)|=(x_0+\varepsilon)-c_\varepsilon^1\leq b^1_\varepsilon-c_\varepsilon^1=|\proj_{e_1} (C_\varepsilon - B^1_\varepsilon)|,$$
we conclude that $B_\varepsilon$ is contained in the closed line segment $[C_\varepsilon, B^1_\varepsilon]$ (see Figure \ref{q1}).

Let
$$q_\varepsilon(x)=\alpha x^2+\beta x+\gamma=\alpha (\varepsilon) x^2+\beta (\varepsilon) x+\gamma(\varepsilon)$$
be the quadratic function such that the support line $H(A_\varepsilon)$ is tangent to $q_\varepsilon$ at $A_\varepsilon$, and the support line $H(B_\varepsilon)$ is tangent to $q_\varepsilon$ at $B^1_\varepsilon$ (see   Figure \ref{q1}). Since $B^1_\varepsilon= (b^1_\varepsilon, \tilde{b}^1_\varepsilon)$ lies on the graph of $q_\varepsilon$, it follows that $\tilde{b}^1_\varepsilon= q_\varepsilon(b^1_\varepsilon)$. 

A simple calculation using \eqref{eq19} shows that as $\varepsilon\rightarrow 0$
\begin{align}\label{eq20}
q_\varepsilon(x)\to \dfrac{u''(x_0)}{2}x^2+(u'(x_0)-x_0 u''(x_0))x+(u(x_0)-x_0u'(x_0)+\dfrac{x_0^2}{2} u''(x_0))\\
\nonumber =q_{x_0}(x).
\end{align}
Since $q_\varepsilon$ is  convex, the point  $B_\varepsilon$ does not lie in the interior of epi($q_\varepsilon$), and the line segment $[B_\varepsilon,B^1_\varepsilon]$ is tangent to $q_\varepsilon$ at $B^1_\varepsilon$. Now, let $B^2_\varepsilon=(b^2_\varepsilon,q_\varepsilon(b^2_\varepsilon))$ be the second point on boundary of  epi($q_\varepsilon$) such that the segment $[B^2_\varepsilon, B_\varepsilon]$ is tangent to $q_\varepsilon$ at $B^2_\varepsilon$, and let
$$T^1_\varepsilon= \conv\{A_\varepsilon,C_\varepsilon,B^1_\varepsilon\} \quad \mbox{and} \quad T^2_\varepsilon=\conv\{A_\varepsilon,C_\varepsilon,B^2_\varepsilon\}.$$
See  Figure \ref{q1}.

\begin{figure}[H]
\centering
\includegraphics[scale=0.23]{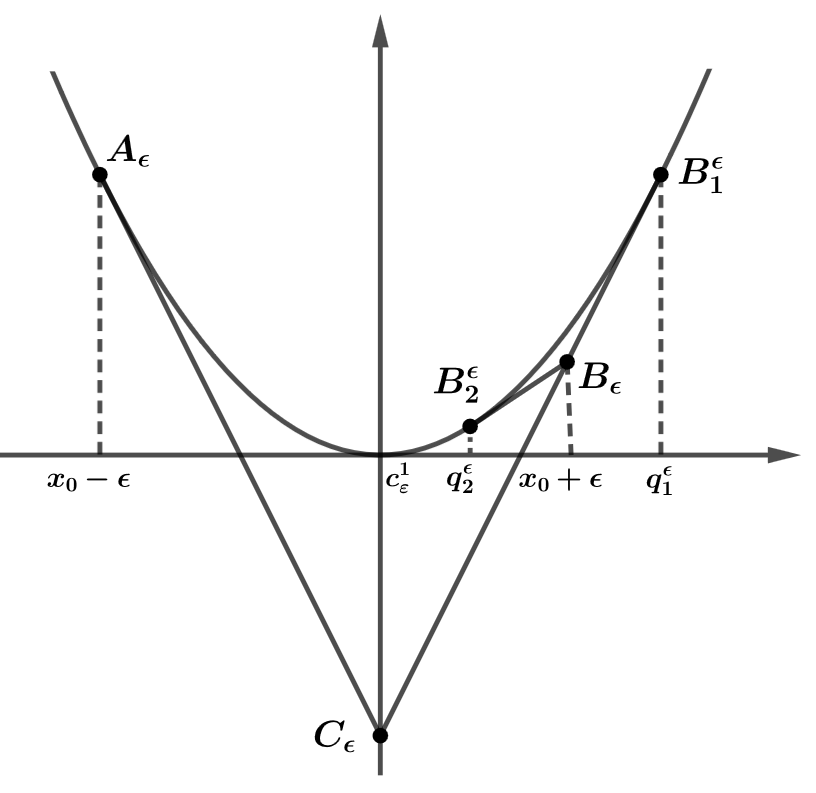}
\caption{Case $u''(x_0)>0$.}\label{q1}
\end{figure}

Now it is sufficient to define
$$ S_T^\varepsilon= (\mbox{epi}(q_\varepsilon) \cap T^2_\varepsilon)\cup \conv\{A_\varepsilon,B^2_\varepsilon,B_\varepsilon\}$$
and
\begin{align}
 v_T^\varepsilon=(q_\varepsilon+\I_{[x_0-\varepsilon,b^2_\varepsilon]})\wedge (l_{[B^2_\varepsilon, B_\varepsilon]}+\I_{[b^2_\varepsilon, x_0+\varepsilon]}).\label{eq1}
\end{align}
Therefore, $S_T^\varepsilon\subset T_\varepsilon$, $S_T^\varepsilon$ is  a support triangle of $v_T^\varepsilon$ and $v_T^\varepsilon\in P\LQ(\mathbb{R})$.
\medskip

It remains to show  item $(iii)$. Using that  $l_{[B^2_\varepsilon, B_\varepsilon]}$ and  $l_{[B_\varepsilon, B^1_\varepsilon]}$ are tangent to $q_\varepsilon$, we obtain
\begin{align*}
    b^2_\varepsilon &= (x_0+\varepsilon)-\sqrt{(x_0+\varepsilon)^2-\alpha^{-1}(u(x_0+\varepsilon)-\beta(x_0+\varepsilon)-\gamma)}\\
    b^1_\varepsilon  &= (x_0+\varepsilon)+\sqrt{(x_0+\varepsilon)^2-\alpha^{-1}(u(x_0+\varepsilon)-\beta(x_0+\varepsilon)-\gamma)},
\end{align*}
and since $u(x_0+\varepsilon)\to q_{x_0}(x_0)$ as $\varepsilon\rightarrow 0$, we get 
\begin{align*}
\lim_{\varepsilon\rightarrow 0} \frac{b^1_\varepsilon-b^2_\varepsilon}{b^1_\varepsilon+\varepsilon-x_0}=0.  
\end{align*}

By \eqref{eq1}, and using that $Z$ vanishes on indicator functions of closed intervals and is dually epi-translation invariant, we have 
$$Z(v_T^\varepsilon)= Z(q_\varepsilon+\I_{[x_0-\varepsilon,b^2_\varepsilon]}).$$
Therefore, by \eqref{eq9}, we obtain  
\begin{align*}
Z(v_T^\varepsilon)& = \dfrac{b^2_\varepsilon+\varepsilon-x_0}{2m}Z(q_\varepsilon+\I_{[-m,m]})\\
&=   \dfrac{b^1_\varepsilon+\varepsilon-x_0}{2m}\left(Z(q_\varepsilon+\I_{[-m,m]})- \dfrac{b^1_\varepsilon-b^2_\varepsilon}{b^1_\varepsilon+\varepsilon-x_0}Z(q_\varepsilon+\I_{[-m,m]})\right).
\end{align*}
For every $\eta>0$ and for  sufficiently small $\varepsilon$, this gives the inequality
\begin{align}\label{eq5.22}
\dfrac{b^1_\varepsilon+\varepsilon-x_0}{2m}\left(Z(q_\varepsilon+\I_{[-m,m]}) -\eta\right)\leq Z(v_T^\varepsilon).   
\end{align}

Next, consider the triangle  $T^1_\varepsilon$, which is a support triangle of the function
$$(u+\I_{[x_0-\varepsilon,x_0+\varepsilon]})\wedge (l_{[B_\varepsilon,B^1_\varepsilon]}+\I_{[x_0+\varepsilon,b^1_\varepsilon]}).$$
This  function belongs to $\LC\NN$, and $T^1_\varepsilon$ is also a support triangle of the quadratic function $q_\varepsilon$.  Let $y_1,\dots, y_n$ be such that 
$$|y_i|=b^1_\varepsilon+\varepsilon-x_0 \quad \mbox{and} \quad n\leq\frac{2m}{b^1_\varepsilon+\varepsilon-x_0}<n+1.$$  
Define the affine map $\phi_i(x,y):\mathbb{R}\times\mathbb{R}\rightarrow \mathbb{R}\times\mathbb{R}$ as
\begin{align}\label{affine}
 \phi_i(x,y)= (x+y_i, \alpha y_i^2+\beta y_i+2\alpha y_ix+y)= (x+y_i, y+\psi_{y_i}(x)),  
\end{align}
for every $i=1,\dots, n$, where $\psi_{y_i}(x)= \alpha y_i^2+\beta y_i+2\alpha y_ix$. Note that $\psi_{y_i}$ is an affine function, $\phi_i(x,q_\varepsilon(x))=(x+y_i, q_\varepsilon(x+y_i))$ and, since $\phi_i$ is a $C^1$ function, it maps support triangles to support triangles.
Specifically, for each $i$, $T_i=T(\tau_{y_i}^{-1}(x_0-\varepsilon), \tau_{y_i}^{-1}(b^1_\varepsilon))$ is also a support triangle of $q_\varepsilon(x)=\alpha x^2+\beta x+\gamma$ for all $i=1,\dots, n$, and the $T_i's$ have pairwise disjoint interiors. Here,  $\tau_{y_i}(x)= x+y_i$ for every $i=1,\dots, n$.

Define 
$$u_i(x+y_i)= (u+\I_{[x_0-\varepsilon,x_0+\varepsilon]})(x) \wedge (l_{[B_\varepsilon,B^1_\varepsilon]}+\I_{[x_0+\varepsilon,b^1_\varepsilon]})(x)+\psi_{y_i}(x)$$
for every $i=1,\dots, n$, and let
\begin{align}\label{eq23}
v_\varepsilon= \bigwedge_{i=1}^n u_i\wedge (q_\varepsilon+\I_{[-m,m]}).
\end{align}
By construction,  $v_\varepsilon$ is a convex function in $\LC\NN$ and

\begin{align*}
Z(v_\varepsilon)\geq  n Z(u+\I_{[x_0-\varepsilon,x_0+\varepsilon]}).
\end{align*}
Furthermore, by \eqref{eq20}, $v_\varepsilon$ is $\tau$-convergent to $q_{x_0}+\I_{[-m,m]}$ as $\varepsilon\rightarrow 0$. 
 
By $\tau$-upper semicontinuity of  $Z$, we then have
\begin{align*}
Z(q_{x_0}+\I_{[-m,m]})& \geq \limsup_{\varepsilon\rightarrow 0} Z(v_\varepsilon)\\
& \geq  \limsup_{\varepsilon\rightarrow 0}\frac{2m}{b^1_\varepsilon+\varepsilon-x_0}Z(u+\I_{[x_0-\varepsilon,x_0+\varepsilon]}).
\end{align*}

Since  the function  $\zeta(2\alpha)=\frac{1}{2m}Z(2\alpha q+\I_{[-m,m]})$ is contained in $\conc$, in particular, it is continuous and by \eqref{eq20}, we get for every $\eta>0$
\begin{align*}
Z(u+\I_{[x_0-\varepsilon,x_0+\varepsilon]})\leq \frac{b^1_\varepsilon+\varepsilon-x_0}{2m}\left(Z(q_\varepsilon+\I_{[-m,m]})+\eta\right)
\end{align*}
for sufficiently small $\varepsilon$. Furthermore, note that
\begin{align*}
b^1_\varepsilon - (x_0+\varepsilon)&\leq \frac{b^1_\varepsilon-(x_0-\varepsilon)}{2} = \varepsilon +\frac{b^1_\varepsilon - (x_0+\varepsilon)}{2}.  
\end{align*}
These last inequalities and \eqref{eq5.22}  now imply that
\begin{align*}
Z(u& +\I_{[x_0-\varepsilon,x_0 +\varepsilon]}) \leq Z(v_T^\varepsilon)+\frac{b^1_\varepsilon+\varepsilon-x_0}{2m}2\eta\\
&\leq Z(v_T^\varepsilon)+\dfrac{2\varepsilon+ \sqrt{(x_0+\varepsilon)^2-\alpha^{-1}(u(x_0+\varepsilon)-\beta(x_0+\varepsilon)-\gamma)}}{2m}2\eta\\
&\leq Z(v_T^\varepsilon)+\dfrac{4\varepsilon}{2m}2\eta
\end{align*}
for  sufficiently small $\varepsilon$.  In the last inequality,  we used the  simple estimate  $$\sqrt{(x_0+\varepsilon)^2-\alpha^{-1}(u(x_0+\varepsilon)-\beta(x_0+\varepsilon)-\gamma)}\leq 2\varepsilon.$$ 
\noindent Finally, setting $\eta=\frac{\rho m}{4}$  shows that item $(iii)$ holds for  sufficiently small $\varepsilon>0$.
\end{proof}

In the case $u''(x_0)=0$ for $x_0\in N$, a similar construction applies, the main difference is that the approximating function $v_T$ is affine on a closed interval.

\begin{lema}\label{lemma2a}
Let  $x_0\in N$ be such that $u''(x_0)=0$. For every $\rho, \delta>0$,  there exist a support triangle $T=T(x,y)$ of $u$, a convex body $S_T\subset \mathbb{R}^2$, and a function  $v_T\in P\LQ(\mathbb{R})$ such that:
\begin{enumerate}
\item [(i)]$(x_0,u(x_0))\in T$ and $0< y-x <\delta$;
\item[(ii)]$S_T\subset T$ and $S_T$ is a support triangle of $v_T$;
\item [(iii)]$Z(u+\I_{[x,y]})\leq Z(v_T)+\frac{\rho}{2} (y-x)$.
\end{enumerate}
\end{lema}
\begin{proof}
Let $T_\varepsilon$ be the support triangle of $u$ with endpoints $A=(x_0,u(x_0))$ and $B_\varepsilon=(x_0+\varepsilon,u(x_0+\varepsilon))$. Let $S_T^\varepsilon=T_\varepsilon$. Then $(i)$ holds for $\varepsilon$ sufficiently small. For every $\alpha>0$, define the quadratic function
$$\tilde{q}(x)=\alpha x^2+(u'(x_0)-2\alpha x_0)x+(u(x_0)+\alpha x_0^2-u'(x_0)x_0)$$
which satisfies the following properties: $A$ is a point of the graph of $\tilde{q}$ and  the epigraph  of $\tilde{q}$ is locally contained in the epigraph of $u$ near $x_0$.

Since $Z$ is a $\tau$-upper semicontinuous, simple, and non-negative  valuation that vanishes on indicator functions of closed intervals, we have
$$\limsup_{\alpha \rightarrow 0} Z(\tilde{q}+\I_{[-m,m]})= 0.$$
Therefore,  for $\alpha$ sufficiently small
\begin{align}\label{eq25}
Z(\tilde{q}+\I_{[-m,m]})\leq \frac{\rho m}{4}.
\end{align}

Next, let $C_\varepsilon$ be the point on the support line to $u$ at $x_0$ such that the first coordinate of  $C_\varepsilon$ is $x_0+\varepsilon$, and let $B^1_\varepsilon=(b^1_\varepsilon,\tilde{q}(b^1_\varepsilon))$ be the point on the graph of $\tilde{q}$ where the line segment $l_{[C_\varepsilon, C^1_\varepsilon]}$ is tangent to the graph of $\tilde{q}$.
\begin{figure}
\centering
\includegraphics[scale=0.18]{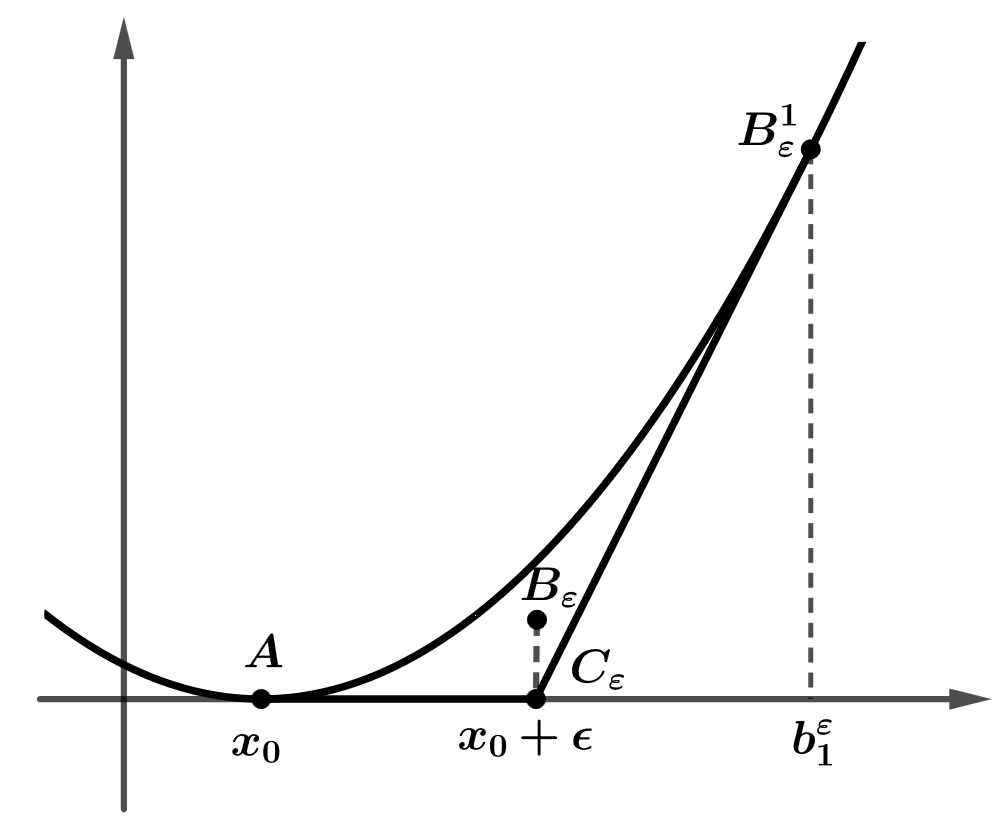}
\caption{ Case $u''(x_0)=0$.  }
\end{figure}

Then the triangle 
$$T^1_\varepsilon=\conv\{A,C_\varepsilon,B^1_\varepsilon\}$$
is a support triangle of $\tilde{q}$. Since epi($\tilde{q}$) is locally contained in epi($u$), the support line of $u$ at $B_\varepsilon$ does not intersect $\tilde{q}$ for  sufficiently  small $\varepsilon> 0$. Therefore, 
$$(u+\I_{[x_0,x_0+\varepsilon]})\wedge (l_{[B_\varepsilon,B^1_\varepsilon]}+\I_{[x_0+\varepsilon,b^1_\varepsilon]})$$
is a convex function in $\LC\NN$, and $T^1_\varepsilon$ is also a support triangle of this function. Using a similar construction as in  \eqref{eq23}, define $v_\varepsilon$ by
\begin{align*}
v_\varepsilon= \bigwedge_{i=1}^n u_i\wedge (\tilde{q}+\I_{[-m,m]}),
\end{align*}
where  $n\leq \frac{2m}{b^1_\varepsilon-x_0}<n+1$, and  $\psi_{y_1},\dots, \psi_{y_n}$ are the  affine functions associated with the $y_1,\dots, y_n$ and $\tilde{q}+\I_{[-m,m]}$, respectively, (see \eqref{affine})
and 
\begin{align*}
u_{i}(x+y_i)=(u+\I_{[x_0,x_0+\varepsilon]})(x) \wedge (l_{[B_\varepsilon,B^1_\varepsilon]}+\I_{[x_0+\varepsilon,b^1_\varepsilon]})(x)+\psi_{y_i}(x).    
\end{align*}

Then, as $\varepsilon\rightarrow 0$, $v_\varepsilon$ is $\tau$-convergent to  $\tilde{q}+\I_{[-m,m]}$,  and  we have 
\begin{align*}
Z(v_\varepsilon)& \geq  nZ(u+\I_{[x_0,x_0+\varepsilon]}).
\end{align*} 
Since $Z$ is $\tau$-upper semicontinuous,  for every $\eta>0$, we obtain the following inequality for $\varepsilon$ sufficiently small,
\begin{align}\label{eq.}
\frac{2m}{b^1_\varepsilon-x_0}Z(u+\I_{[x_0,x_0+\varepsilon]})\leq Z(\tilde{q}+\I_{[-m,m]}) +\eta.   
\end{align}

Since $C_\varepsilon$ lies on the support line  to $\tilde{q}$ containing $A$, and the support line containing $B^1_\varepsilon$, we have that $b^1_\varepsilon=x_0+2\varepsilon$. Thus, replacing $b^1_\varepsilon$ into inequality \eqref{eq.}, we obtain
\begin{align}\label{eta}
Z(u+\I_{[x_0, x_0+\varepsilon]})\leq \left(\frac{\eta}{m}+\frac{1}{m}Z(\tilde{q}+\I_{[-m,m]})\right)\varepsilon .    
\end{align}
By \eqref{eq25}, it follows that
\begin{align*}
Z(u+\I_{[x_0, x_0+\varepsilon]})&\leq  \left(\frac{\eta}{m} +\frac{\rho}{4}\right) \varepsilon.
\end{align*}
Taking $\eta=\dfrac{\rho m}{4}$, we get
$$Z(u+\I_{[x_0, x_0+\varepsilon]})\leq    \frac{\rho}{2} \varepsilon.$$

Finally, since
$$v_T^\varepsilon= l_{[A, B_\varepsilon]}+\I_{[x_0,x_0+\varepsilon]}$$
is an  affine function on $[x_0,x_0+\varepsilon]$,  we conclude that conditions $(ii)$ and $(iii)$ hold for sufficiently small $\varepsilon$.
\end{proof}

Further, we need the following result.

\begin{lema}\label{singular}
There is a constant $c_{L_u}$ such that
\begin{align*}
    Z(u+\I_J)\leq c_{L_u} V_1(J) 
\end{align*}
for every closed interval $J\subseteq \dom u=[-m,m]$,  where $c_{L_u}$ depends only  on  the Lipschitz constant $L_u$ of $u$   and $m>0$.
\end{lema}

\begin{proof}
Let  $L_u$ denote  the Lipschitz constant of $u$. There exists a quadratic function
$$q_{x_0}(x)= x^2+\beta(x_0)x+\gamma(x_0)$$
such that $L_u \,q_{x_0}(x_0)= u(x_0)$ and $L_u\, q_{x_0}'(x_0)= u_+'(x_0)$, where $u_+'(x_0)$ denotes the right-hand derivative of $u$ at $x_0$. This implies that the epigraph of $L_u\, q_{x_0}$ is locally contained in the epigraph of $ u$. Let $T_\varepsilon$ be the support triangle of $u$ with endpoints $A=(x_0,u(x_0))$ and $B_\varepsilon(x_0+\varepsilon, u(x_0+\varepsilon))$,  where $\varepsilon>0$ is sufficiently small. As in  Lemma \ref{lemma2a}, define $B_\varepsilon^1=(b_\varepsilon^1, q_{x_0}(b_\varepsilon^1))$ and 
\begin{align*}
v_\varepsilon= \bigwedge_{i=1}^n u_i\wedge (L_u\, q_{x_0}+\I_{[-m,m]})
\end{align*}
where  $n\leq \frac{2m}{b^1_\varepsilon-x_0}<n+1, \psi_{y_1},\dots, \psi_{y_n}$ are the  affine functions associated with $y_1,\dots, y_n$ and $L_u\, q_{x_0}+\I_{[-m,m]}$, respectively, and  
\begin{align*}
u_{i}(x+y_i)=(u+\I_{[x_0,x_0+\varepsilon]})(x) \wedge (l_{[B_\varepsilon,B^1_\varepsilon]}+\I_{[x_0+\varepsilon,b^1_\varepsilon]})(x)+\psi_{y_i}(x).    
\end{align*}
Then $v_\varepsilon$ is $\tau$-convergent to  $L_u\, q_{x_0}+\I_{[-m,m]}$ as $\varepsilon\rightarrow 0$. Moreover, for $\eta=1$, we obtain, as in~\eqref{eta}, the inequality 
\begin{align}
 Z(u+\I_{[x_0, x_0+\varepsilon]}) &\leq   \left(\frac{1}{m}+\frac{Z(L_u\, q_{x_0}+\I_{[-m,m]})}{m}\right)\varepsilon\label{uniforme}
\end{align}
for  sufficiently small $\varepsilon$. Note that $Z(L_u\, q_{x_0}+\I_{[-m,m]})$  depends only on $L_u$ and $m$, since $Z$ is dually epi-translation invariant. Therefore, we can partition $J$ into finitely many intervals  on which \eqref{uniforme} holds.
\end{proof}


Combining Lemma~\ref{lemma2}, Lemma~\ref{lemma2a}, and Lemma~\ref{singular}, we are now able to prove \eqref{newe}, which will be used in the proof of Proposition~\ref{prop2}.

\begin{prop}\label{prop1}
If $Z:\LC\NN\to \mathbb{R}$ is a simple, $\tau$-upper semicontinuous, and translation  invariant valuation that vanishes on indicator functions of closed intervals, then
\begin{align*}
Z(u)=\sup\left\{\limsup_{n\rightarrow +\infty }Z(v_n)\mid \ v_n\in P\LQ(\mathbb{R}), v_n\rightarrow u\right\}.
\end{align*}
\end{prop}
\begin{proof}
Since $u\in\LC\NN$, the set $N\subseteq \dom u$ of points where $u$ is twice differentiable is such that
$$V_1(N)= V_1(\dom u).$$

\noindent  By Lemma \ref{lemma2} and Lemma \ref{lemma2a}, the sets 
$$\{x\in N\mid \ (x,u(x))\in T\},$$  
where $T$ is a support triangle of $u$ satisfying the conditions of Lemma \ref{lemma2} or Lemma \ref{lemma2a}, form a Vitali class for $N$,  and this remains true when we restrict to sets $T$   such that $V_1(\proj_{e_1}T)\leq \delta$ for some small  $\delta>0$. Let $\eta>0$ such that  
$$\eta\leq \delta \quad \text{and} \quad \eta\leq\frac{\rho}{2 c_{L_u}}V_1(\dom u),$$ 
where $c_{L_u}$ is given by Lemma \ref{singular}.  Then,  by Vitali's Theorem \ref{VCT}, we can choose  support triangles $T_1,\dots,T_n$  such that
\begin{align*}
V_1(\dom u)=V_1(N)\leq \sum_{i=1}^n V_1(P_{e_1} T_i)+\eta,    
\end{align*}
where the closed intervals $P_{e_1} T_1, \dots, P_{e_1} T_n$ have pairwise disjoint interiors. We also  choose closed intervals $J_1,\dots, J_k$ such that the sets  $P_{e_1} T_1, \dots, P_{e_1} T_n, J_1,\dots, J_k$  have pairwise disjoint interiors, and such that   $\dom u$ can be decomposed as 
$$P_{e_1} T_1\cup \cdots\cup P_{e_1} T_n\cup  J_1\cup \cdots \cup J_k.$$

Define
$$v_T=v_{T_1}\wedge\cdots \wedge v_{T_n}\wedge (l_1+\I_{J_1})\wedge \cdots \wedge (l_k+\I_{J_k}),$$
where the $l_i's$ are piecewise affine functions, and $v_T$ is a convex function  in $\LC\NN $ whose domain coincides with that of $u$.  Our construction using support triangles implies that 
\begin{align*}
Z(v_T) & = \sum_{i=1}^n Z(v_{T_i}).
\end{align*}
Therefore,  by Lemma  \ref{lemma2}, Lemma \ref{lemma2a}, and Lemma \ref{singular}, and using that $Z$ is a simple valuation, we conclude that for every $\rho>0
$,
\begin{align*}
Z(u)&= \sum_{i=1}^n Z(u+\I_{P_{e_1} T_i})+\sum_{s=1}^k Z(u+\I_{J_s})\\
&\leq \sum_{i=1}^n \left(Z(v_{T_i})+\frac{\rho}{2}V_1(P_{e_1}T_i)\right)+c_{L_u}\sum_{s=1}^k V_1(\dom u\cap J_s) \\
& \leq Z(v_T)+\frac{\rho}{2} \sum_{i=1}^n V_1(P_{e_1}T_i)  + c_{L_u}\eta \\
& \leq Z(v_T)+\frac{\rho}{2} \sum_{i=1}^n V_1(P_{e_1}T_i) + \frac{\rho}{2} V_1(\dom u) \\
&\leq  Z(v_T)+\rho V_1(\dom u).
\end{align*}
Since $\rho$ is arbitrary, we conclude the proof of the proposition.
\end{proof}

By Proposition \ref{prop1} we get the following result. 

\begin{prop}\label{prop2}
Let $\zeta \in \conc$. Then there exists a unique functional $Z: \LC\NN \to \mathbb{R}$ such that:
\begin{enumerate}
\item [(i)] $Z$ is $\tau$-upper semicontinuous;
\item [(ii)] $Z$ is a simple, translation and dually epi-translation invariant valuation that vanishes on indicator functions of closed intervals;
\item [(iii)]  $Z(aq +\I_{[-m,m]})=2m\zeta(a)$ for all $a>0$ and $m>0$, where $q(x)=\frac{x^2}{2}$.
\end{enumerate}
\end{prop}

\begin{proof}
Let  $Z: \LC\NN\rightarrow[0,+\infty)$ be a valuation that satisfies the conditions $(i)-(iii)$.  By Lemma \ref{lemma1},   $\zeta\in\conc$. Since $Z$ is a simple valuation and,  by  \eqref{eq4}, it is determined by $\zeta$ on piecewise linear-quadratic functions, we conclude that $Z(u)$ is determined by   $\zeta$ for every $u\in P\LQ(\mathbb{R})$. Furthermore, by  Proposition~\ref{prop1}, we conclude that $Z$ is uniquely determined by $\zeta$. 
\end{proof}

As a consequence of Lemma~\ref{lemma1} and Proposition~\ref{prop2}, we have the following proposition.

\begin{prop}\label{prop}
Let $Z: \LC\NN\rightarrow \mathbb R$ be a simple,  $\tau$-upper semicontinuous, translation and dually epi-translation invariant valuation which vanishes on indicator functions of closed intervals. Then there is a function $\zeta\in\conc$  such that
\begin{align*}
    Z(u)= \int_{\dom u} \zeta (u''(x))\dif x
\end{align*}
for every $u\in\LC\NN$.
\end{prop}
\begin{proof} Let $\zeta \in\conc$  be given by  \eqref{eq4}. Applying Lemma \ref{concave} and  Theorem~\ref{mainteo} for $n=1$,  we obtain that the functional $Z_{\zeta}: \LC\NN\rightarrow[0,+\infty)$ defined by 
\begin{align*}
    Z_{\zeta}(u)= \int_{\dom u} \zeta(u''(x))\dif x
\end{align*}
is a simple, $\tau$-upper semicontinuous,  translation and dually epi-translation invariant valuation. Moreover, it vanishes on the indicator functions of closed intervals, and for $q(x)= \frac{x^2}{2}$  and $a\geq 0$ satisfies 
\begin{align*}
 Z_{\zeta}(aq+\I_{[-m,m]}) = 2m\zeta(a).  
\end{align*}
Therefore, by Proposition \ref{prop2}, we conclude the proof.
\end{proof}

Finally, using Proposition~\ref{prop}, we prove Theorem~\ref{maintheorem}.

{\begin{proof}[\textit{Proof of Theorem \ref{maintheorem}}]
Let $x_0\in\mathbb R$. Since $Z$ is a translation invariant valuation, we have $Z(\I_{\{x_0\}})= c_0$ for some constant $c_0\in\mathbb{R}$ that does not depend on $x_0$. Define
\begin{align*}
    Z_0(u)=Z(u)-c_0.
\end{align*}
Then $Z_0$ is a simple, $\tau$-upper semicontinuous, translation, and dually epi-translation invariant valuation.
By Lemma \ref{non-positive}, there exists a constant $c_1\in\mathbb{R}$ such that 
$Z_0(\I_J)=c_1V_1(J)$ for every closed interval $J\subset\mathbb{R}$. Furthermore, by Lemma \ref{volume2} and  Lemma \ref{volume1}, the functional
$Z_1(u)=Z_0(u)-c_1V_1(\dom u)$ is a simple, $\tau$-upper semicontinuous, translation, and dually epi-translation invariant valuation, which vanishes on indicator functions of closed intervals. The proof now follows from  Proposition \ref{prop}, which guarantees the existence of a function $\zeta\in\conc$  such that
\begin{align*}
    Z_1(u)= \int_{\dom u} \zeta (u''(x))\dif x
\end{align*}
for every $u\in \LC\NN$.
\end{proof}

\section*{Acknowledgments}
The author thanks Monika Ludwig for  insightful discussions and careful reading of the manuscript. The author also thanks Fabian Mussnig for helpful discussions and suggestions. This project was supported, in part, by the Austrian Science Fund (FWF) Grant-DOI: 10.55776/P34446 and Grant-DOI: 10.55776/P37030. For open access purposes, the author has applied a CC BY public copyright license to any author-accepted manuscript version arising from this submission.

\vspace{0.5cm}

\begin{thebibliography}{xx}

\bibitem{aczel} J. Aczel, \textit{Lectures on functional equations and their applications}. Academic Press, New York (1966).


\bibitem{alek} A. D. Aleksandrov, ``Almost everywhere existence of the second differential of a convex function and some properties of convex surfaces connected with it'' (in Russian), \textit{Uchenye Zapiski Leningrad. Gos. Univ.}, Math. Ser. 6 (1939), 3–35.



\bibitem{alesker} S. Alesker, ``Continuous rotation invariant valuations on convex sets''. \textit{Ann. of Math.} 149 (1999), 977–1005.


\bibitem{integral} F. M. Baêta and M. Ludwig, ``On the semicontinuity of functionals on function spaces''. arXiv:2509.17426 (2025).



\bibitem{beck} A. Beck, \textit{First-order methods in optimization}. SIAM (2017).


\bibitem{Vitor} V. Balestro, H. Martini and R. Teixeira, \textit{Convexity from the geometric point of view}. 1st ed., Birkhäuser Cham (2024).


\bibitem{Blaschke1937}
W.~Blaschke,
 \textit{Vorlesungen über Integralgeometrie.} H. 2,
 Teubner, Berlin, 1937.




\bibitem{irmn} A. Colesanti, M. Ludwig and F. Mussnig, ``Valuations on convex functions''. \textit{Int. Math. Res. Not.} (2019), 2384–2410.



\bibitem{1} A. Colesanti, M. Ludwig and F. Mussnig, ``A homogeneous decomposition theorem for valuations on convex functions''. \textit{J. Funct. Anal.} 279 (2020), 108573.



\bibitem{H1} A. Colesanti, M. Ludwig and F. Mussnig, ``The Hadwiger theorem on convex functions'', I. \textit{Geom. Funct. Anal.} 34 (2024), 1839–1898.

\bibitem{topology} A. Colesanti, D. Pagnini, P. Tradacete, and I. Villanueva, ``A class of invariant valuations on Lip($S^{n-1}$)''. \textit{Adv. Math.} 366 (2020), 107069.

\bibitem{topology2} A. Colesanti, D. Pagnini, P. Tradacete, and I. Villanueva, ``Continuous valuations on the space of Lipschitz functions on the sphere''. \textit{J. Funct. Anal.} 280(4) (2021), 108873.

\bibitem{topology3} A. Colesanti, J. Knoerr, and D. Pagnini, ``The homogeneous decomposition of dually translation invariant valuations on Lipschitz functions on the sphere''. arXiv:2401.05913v1 (2024).


\bibitem{2} K. Falconer, \textit{The geometry of fractal sets}. Cambridge University Press, Cambridge (1985).



\bibitem{99} A. Figalli, \textit{The Monge--Ampère equation and its applications}. Zürich Lectures in Advanced Mathematics. European Mathematical Society (EMS), Zürich (2017).


\bibitem{haberl} C. Haberl, ``Star body valued valuations''. \textit{Indiana Univ. Math. J.} 58 (2009), 2253–2276.

\bibitem{haberl2} C. Haberl and M. Ludwig, ``A characterization of $L_p$ intersection bodies''. \textit{Int. Math. Res. Not.} 17 (2006), Art. ID 10548, 29 pp.


\bibitem{81} H. Hadwiger, \textit{Vorlesungen über Inhalt, Oberfläche und Isoperimetrie}. Springer, Berlin, 1957.




\bibitem{klain} D. A. Klain and G. Rota, \textit{Introduction to geometric probability}. Cambridge University Press, Cambridge, UK (1997).


\bibitem{71'} M. Ludwig, ``A characterization of affine length and asymptotic approximation of convex discs''. \textit{Abh. Math. Semin. Univ. Hamb.} 69 (1999), 75–88.


\bibitem{71} M. Ludwig, ``Upper semicontinuous valuations on the space of convex discs''. \textit{Geom. Dedicata} 80 (2000), 263–279.


\bibitem{ludwig2} M. Ludwig, ``Valuations on Sobolev spaces''. \textit{Am. J. Math.} 134 (2012), 827–842.


\bibitem{monika} M. Ludwig and M. Reitzner, ``A characterization of affine surface area''. \textit{Adv. Math.} 147 (1999), 138–172.


\bibitem{L} M. Ludwig, ``Geometric valuation theory''. In \textit{European Congress of Mathematics}, EMS Press, Berlin (2023), pp. 93–123.


\bibitem{lutwak} E. Lutwak, ``Extended affine surface area''. \textit{Adv. Math.} 85 (1991), 39–68.

\bibitem{M1} F. Mussnig, ``Volume, polar volume and Euler characteristic for convex functions''. \textit{Adv. Math.} 344 (2019), 340–373.

\bibitem{mullen} P. McMullen, ``Valuations and dissections''. In P. M. Gruber and J. M. Wills (Eds.), \textit{Handbook of convex geometry B}, North-Holland, Amsterdam (1993), pp. 933–988.


\bibitem{mullen2} P. McMullen and R. Schneider, ``Valuations on convex bodies''. In \textit{Convexity and its applications}, \textit{Birkhäuser}, Basel (1983), pp. 170–247.


\bibitem{76} R. T. Rockafellar and R. J.-B. Wets, \textit{Variational analysis}. Grundlehren der mathematischen Wissenschaften, 3rd ed., vol. 317, Springer-Verlag, Berlin (2009).


\bibitem{floating} C. Schütt and E. M. Werner, ``The convex floating body''. \textit{Math. Scand.} 66 (1990), 275–290.


\bibitem{E} C. Schütt and E. M. Werner, ``Affine surface area''. In \textit{Harmonic analysis and convexity}, Adv. Anal. Geom., De Gruyter, Berlin (2023), pp. 427–444.


\bibitem{TW2} N. S. Trudinger and X.-J. Wang, ``The Bernstein problem for affine maximal hypersurfaces''. \textit{Invent. Math.} 140 (2000), 399–422.



\bibitem{TW6} N. S. Trudinger and X.-J. Wang, ``The affine Plateau problem''. \textit{J. Amer. Math. Soc.} 18 (2005), 253–289.


\bibitem{TW8} N. S. Trudinger and X.-J. Wang, ``The Monge--Ampère equation and its geometric applications''. \textit{Geom. Anal.}, International Press, Vol. I (2008), pp. 467–524.

\bibitem{TS} A. Tsang, ``Valuations on $L^p$ spaces''. \textit{Int. Math. Res. Not.} 20 (2010), 3993–4023.


\end{thebibliography}
\end{document}